\documentclass[reqno,11pt]{amsart}
\usepackage{amsmath,amssymb,amsfonts}

\usepackage{tikz}
\usetikzlibrary{matrix,arrows}
\usepackage{verbatim}

\usepackage[pagebackref,colorlinks,citecolor={red!50!black},linkcolor={blue!80!black}]{hyperref}

\renewcommand*{\backref}[1]{}
\renewcommand*{\backrefalt}[4]{%
	\scriptsize%
	{
		\ifcase #1 (\textcolor{red}{Uncited.})%
		\or (Cited\ on p.~#2)%
		\else (Cited\ on pp.~#2)%
		\fi%
	}
}

\usepackage{orcidlink}

\usepackage{fourier}
\usepackage{mathtools}
\usepackage[utf8]{inputenc}
\usepackage[varqu,varl]{inconsolata}

\title{Homotopy limits of complexes}

\author[L. Alonso]{Leovigildo Alonso Tarr\'{\i}o \orcidlink{0000-0002-6896-0652}}
\address[L. A. T.]{CITMAGA\\
	Departamento de Matem\'a\-ticas\\
	Universidade de Santiago de Compostela\\
	E-15782  Santiago de Compostela, Spain}
\email{leo.alonso@usc.es}

\author[R. Alvite]{Ra\'ul Alvite Paz\'o \orcidlink{0009-0007-1753-7057}}
\address[R. A. P.]{CITMAGA\\
	Departamento de Matem\'a\-ticas\\
	Universidade de Santiago de Compostela\\
	E-15782  Santiago de Compostela, Spain}
\email{raul.alvite.pazo@usc.es}

\author[A. Jerem\'{\i}as]{Ana Jerem\'{\i}as L\'opez \orcidlink{0000-0001-7964-1334}}
\address[A. J. L.]{CITMAGA\\
	Departamento de Matem\'a\-ticas\\
	Universidade de Santiago de Compostela\\
	E-15782  Santiago de Compostela, Spain}
\email{ana.jeremias@usc.es}


\date{\emph{Unreleased}; \emph{typeset}: \today}

\theoremstyle{plain}
\newtheorem{thm}{Theorem}[section]
\newtheorem{lem}[thm]{Lemma}
\newtheorem{cor}[thm]{Corollary}
\newtheorem{prop}[thm]{Proposition}

	\makeatletter
	\newtheorem*{rep@theorem}{\rep@title}
	\newcommand{\newreptheorem}[2]{%
		\newenvironment{rep#1}[1]{%
			\def\rep@title{#2~\ref{##1}}%
			\begin{rep@theorem}}%
			{\end{rep@theorem}}}
	\makeatother
	
	\makeatletter
	\def\namedlabel#1#2{\begingroup
		\def\@currentlabel{#2}%
		\label{#1}\endgroup}
	\makeatother

\newreptheorem{theorem}{Theorem}
\newreptheorem{cor}{Corollary}

\theoremstyle{remark}
\newtheorem*{rem}{Remark}

\theoremstyle{definition}

\newtheorem*{ack}{Acknowledgements}
\newtheorem{cosa}[thm]{}

\numberwithin{equation}{thm}

\newcommand{\CN}{\mathcal{N}}

\newcommand{\CP}{\mathcal{P}}

\newcommand{\SA}{\mathsf{A}}

\newcommand{\SC}{\mathsf{C}}

\newcommand{\SI}{\mathsf{I}}

\newcommand{\SR}{\mathsf{R}}

\newcommand{\ST}{\boldsymbol{\mathsf{T}}}

\newcommand{\CCC}{\boldsymbol{\mathsf{C}}}
\newcommand{\D}{\boldsymbol{\mathsf{D}}}
\newcommand{\K}{\boldsymbol{\mathsf{K}}}
\newcommand{\LL}{\boldsymbol{\mathsf{L}}}
\newcommand{\R}{\boldsymbol{\mathsf{R}}}

\newcommand{\NN}{\mathbb{N}}
\newcommand{\ZZ}{\mathbb{Z}}

\newcommand{\holim}[1]{\begin{array}[t]{c} {\rm holim}\\[-7.5 pt]
		{\longleftarrow} \end{array}}
\newcommand{\hocolim}[1]{\begin{array}[t]{c} {\rm holim}\\[-7.5 pt]
		{\longrightarrow} \end{array}}

\newcommand{\lto}{\longrightarrow}

\newcommand{\epi}{\twoheadrightarrow}

\newcommand{\iso}{\mathrel{\tilde{\to}}}

\DeclareMathOperator{\Hom}{Hom}

\DeclareMathOperator{\rhom}{\R\!\Hom}

\DeclareMathOperator{\Tot}{Tot}
\DeclareMathOperator{\Img}{Im}
\DeclareMathOperator{\cok}{Coker}

\DeclareMathOperator{\h}{H}
\DeclareMathOperator{\id}{id}

\newcommand{\md}{\text{-}\mathsf{Mod}}

\newcommand{\op}{\mathsf{op}}
\newcommand{\ab}{\mathsf{Ab}}
\newcommand{\set}{\mathsf{Set}}

\DeclareMathOperator{\cone}{Cone}

\DeclareMathOperator{\cat}{\mathsf{Cat}}

\newcommand{\ie}{{\it i.e.~}}

\DeclareMathOperator{\pfinmax}{\CP_{\mathsf{f}}^{\mathsf{m}}}

\begin{document}
	
\begin{abstract} 
We propose a notion of homotopy limit in the category of complexes over an abelian category with products by totalizing the classic construction of the Roos' complex that computes derived inverse limits. For complexes of modules over a non-necessarily commutative ring, we show that our construction of homotopy limits computes the derived limit complex, and under an acyclicity hypothesis on the inverse system, we prove that it is quasi-isomorphic to the limit. We further show that, in general, the construction is appropriately dual of the previous construction of homotopy colimits of complexes from \cite{AJS}, and that there also is a dual behavior between derived limits and colimits in derived categories of modules. Finally, we show that colocalizing subcategories are stable for homotopy limits.
\end{abstract}
	
\maketitle
\tableofcontents

\section*{Introduction}
	
One of the distinctive features of derived categories is the absence of limits and colimits, due to the localization procedure involved in their definition. A remedy is to propose a functorial construction that recovers at least the homology of the purported limits or colimits of the diagrams involved. In the topological context, this strategy goes back to Bousfield and Kan \cite{BK}.

The classical paper \cite{BN} proposed a purely homological algebra construction of homotopy co/limits of countable sequential systems as functors on the category of complexes whose homologies compute the homologies of the true co/limits of the diagrams. In \cite{AJS} a purely algebraic construction of the homotopy colimits of a more general diagram of complexes (indexed by an ordered set) was proposed. In the context of Grothendieck categories, \ie, where we assume the existence of a generator and the exactness of filtered colimits (the known condition $\mathsf{AB5}$), their behavior is excellent. The main result that explains the adequacy of the construction is that they are quasi-isomorphic to the actual colimit of complexes \cite[Theorem 2.2]{AJS}. 

A further crucial property is that, given a system of objects that belong to a localizing subcategory, then its homotopy colimit remains in the localizing subcategory \cite[Theorem 3.1]{AJS}. We remind the reader that a localizing subcategory of a triangulated category is a triangulated subcategory stable for (small) coproducts.

It is clear that the case of limits (\textit{i.e.}, inverse limits) is more complicated due to their lack of exactness in the abelian categories of interest like Grothendiek categories. In this paper we propose a construction of homotopy limit by extending the classical construction of Roos \cite{roos-1958} to the context of complexes by totalization. That this construction is reasonable is witnessed by its good properties, analogous to those enjoyed by the previously mentioned construction of homotopy colimits.

Let us discuss with some detail these properties. The main results exhibit the compatibility of our notion with usual limits for the derived category of modules over a general ring (associative and unital), under a mild cardinality hypothesis on the inverse system. The first one needs conditions on the system that ensure acyclicity for limits. It says: 

\begin{reptheorem}{thm:quis-abelian-cat-ab4*-gen}
	Suppose $R\md$ is the category of left $R$-modules for a ring $R$, and let $\Lambda$ be a filtered poset of cardinality \,$\#\Gamma < \aleph_\omega$. If $F \colon \Lambda^\op \to \CCC(R\md)$ is an inverse system which is weakly flabby (degree-wise), then there exists a quasi-isomorphism 
	\[ \varprojlim F \lto \holim{} F.\]
\end{reptheorem}

Further, relaxing the acyclicity condition, this complex computes derived limits. The corresponding statement is the following:

\begin{reptheorem}{thm:in-rmod-holim-is-rlim}
In the previous hypotheses, let $F \colon \Lambda^\op \to \CCC(R\md)$ be \emph{any} inverse system. There exists a natural isomorphism in $\D(R\md)$ 
\[
	\holim{} F \cong \R\varprojlim F.
\]
\end{reptheorem}

An additional result, just under the minimal $\mathsf{AB3}$ condition (on existence of coproducts) on the abelian category under consideration, shows that homotopy limits and colimits are related in a dual way, the precise statement is Theorem \ref{hom-and-hocolim-holim}. This result also shows how our theory is consistent with the notion of homotopy colimit developed in \cite{AJS}.

In derived categories of modules, this last result combines with Theorem \ref{thm:in-rmod-holim-is-rlim} to establish an enriched dual relationship between derived colimits and derived limits. The statement is as follows:

\begin{repcor}{cor:derived-limits-and-colimits}
	Suppose $R\md$ is the category of left $R$-modules for a ring $R$, and let $\Lambda$ be a filtered poset of cardinality \,$\#\Lambda < \aleph_\omega$. For any direct system $G \colon \Lambda \to \CCC(R)$ and $X \in \CCC(R)$, we have the following isomorphism in $\D(R)$
	\[
	\rhom^\bullet(\LL\varinjlim G, X) \cong \R\varprojlim \rhom^\bullet(G(-),X),
	\]
	or, equivalently,
	\[
	\rhom^\bullet(\!\!\!\hocolim{}\!G, X) \cong \holim{} \!\!\!\rhom^\bullet(G(-),X).
	\]
\end{repcor}

All this shows that, at least for the derived category of a ring, our construction gives a reasonable  notion of homotopy limit. A recent paper by Arakawa \cite{Arak} provides a categorical construction of a homotopy limit in terms of the bar construction and justifies the adequacy of the construction by appealing to a certain model category structure on the category of complexes. Our construction is aligned with classical homological algebra methods. Its adequacy relies on the previously mentioned results relating our construction to classical derived limits.

The previous results use in an essential way the exactness of products, a very specific property of the category of modules among Grothendieck categories. The analogous results on more general categories would require subtler conditions on the abelian base category but we will not explore them here.

Under weaker hypotheses, we show in addition that our homotopy limits behave well vis-a-vis  colocalizing subcategories. Let us make this precise. A colocalizing subcategory of a triangulated category is a triangulated subcategory stable for products. Under the conditions $\mathsf{AB3}$, $\mathsf{AB4^*}$ and the existence of a generator for an abelian category $\SA$, we prove that given an inverse system all of whose objects lie in a colocalizing subcategory of $\D(\SA)$, its homotopy limit remains in the colocalizing subcategory. This is the content of Theorem \ref{thm:holim-colocalizing}. This result is in a sense dual to \cite[Theorem 3.1]{AJS}.

Let us now describe the organization of the paper. In section \ref{Section:Conventions} we establish the basic notations and sign conventions. We recall Roos' construction of the complex that computes derived categories following mainly \cite{Ntc}. We also discuss some acyclicity conditions for inverse systems. Some of them arise as variants of the classical Mittag-Leffler condition from \cite{ega3}. Next, we present conditions that arise in the interpretation of an inverse system as an associated sheaf on a topological space naturally attached to the ordered set defining the shape of the diagrams of which we are taking limits and compare to the previously discussed acyclicity conditions.

The next section \ref{Section:homotopy-limit} contains the main results on the characterization of homotopy limits. The basic tool is to decompose the inverse system in an inverse system of \emph{finite} inverse systems with a maximum for which the result is obvious and explicit. In section \ref{Section:duality} we explore the duality between homotopy limits and colimits, and extend it to derived limits and colimits using the results of the previous section. Finally, in section \ref{Section:homotopy-limits-and-colocalizing-subcategories} we establish the fact that colocalizing subcategories are stable under homotopy limits. For this result it is essential to decompose the homotopy colimit complex by a cofiltration that satisfies the Mittag-Leffler condition, a decomposition dual to a filtration that exists for homotopy colimits \cite[Lemma 3.2]{AJS}. Overall, the lack of exactness on limits requires a special care and some vanishing condition in order to achieve the dualization of some properties of homotopy colimits.

\begin{ack}
The authors are grateful to Eduardo Loureiro who made us interested in homotopy limits and gave us an initial push to develop the present theory. 
\end{ack}
	
\section{Notation and conventions}\label{Section:Conventions}
	Let $\SA$ be an abelian category. Denote by $\CCC(\SA)$ the category of (cochain) complexes of objects of $\SA$ and by $\CCC\CCC(\SA)$ the category of double complexes of objects of $\SA$. The convention for this last category is that the objects of $\CCC\CCC(\SA)$ are complexes of complexes. An object of $\CCC\CCC(\SA)$ is a family $\left\{X, d_h, d_v\right\}$ where $X$ is a bigraded object, for $r,s \in \ZZ$ $d_h^{rs} \colon X^{rs} \to X^{r+1 s}$ corresponds to the ``horizontal'' differential, $d_v^{rs} \colon X^{rs} \to X^{r s+1}$ corresponds to the ``vertical'' differential and the two differentials commute, that is, $d_h \circ d_h = 0$, $d_v \circ d_v = 0$ and $d_v \circ d_h = d_h \circ d_v$.
	
	Throughout the text, by a ring $R$ we will mean a possibly non-commutative ring (\textit{i.e.}, associative and unital). We will always consider left $R$-modules and denote the corresponding category by $R\md$. In the event that $\SA = R\md$, we will write $\CCC(R) \coloneqq \CCC(R\md)$, as customary. Similarly, we will also denote the homotopy category by $\K(R) \coloneqq \K(R\md)$ and the derived category by $\D(R) \coloneqq \D(R\md)$.
	
\begin{cosa}\label{totalization}
		\textbf{Totalization.}
		There are two main ways in which an object in $\CCC(\SA)$ can be obtained from a double complex, both of them are often called the total complex of a double complex. Both constructions agree when the double complex is bounded in one of its degrees. In the general case, they differ. We employ the term totalization of a complex instead.
		
		The most usual one is the \textit{totalization by coproducts}, which for a double complex $\{X, d_h, d_v\}$ is defined, in each degree $n \in \ZZ$, by 
		\[ \left(\Tot^{\coprod}(X)\right)^n \coloneqq \coprod_{r+s=n} X^{rs}, \]
		and whose differential is given by 
		\begin{equation}\label{differential-totalization}
			d^n \coloneqq \sum_{r+s = n} d_h^{rs} + (-1)^r d_v^{rs}.
		\end{equation}
		The \textit{totalization by products} of $\{A, d_h, d_v\}$ is definted similarly, that is, 
		\[
			\left(\Tot^{\prod}(X)\right)^n \coloneqq \prod_{p+q=n} X^{pq}
		\] for each $n \in \ZZ$ and its differential is defined following the rule of the totalization by coproducts.
		
		Importantly, in both cases, while the sign involved in the differential comes from the horizontal degree, it acts on the vertical part of the differential. For future reference, we will explicitly describe the differential of the totalization by products after composing with a projection, as we will make use of the universal property of the product repeatedly in Section \ref{Section:homotopy-limit}. Suppose we have totalized by products a double complex $(X, d_h, d_v)$ and thus obtained a complex $(\Tot^{\prod}(X), d)$. Then, in degree $n$ the composite 
		\[
			\prod_{k+l=n} X^{kl} \xrightarrow{d^n} \prod_{r+s=n+1} X^{rs} \xrightarrow{p_{rs}} X^{rs},
		\]
		where $p_{rs}$ is the projection, can be described as
		\begin{equation}\label{differential-totalization-products}
			p_{rs} d^n = d_h^{r-1 s} p_{r-1 s} + (-1)^r d_v^{r s-1} p_{r s-1}.
		\end{equation}
		That is, the expression in Equation \ref{differential-totalization} describes the morphisms going out of the bigraded pieces making up the totalized complex in each degree, while the one above describes the morphisms going in.
		
		It is straightforward to check that totalization (by coproducts or products) is functorial.
	\end{cosa}

	\begin{cosa}\label{convention-internal-hom}
		\textbf{Sign convention for the internal $\Hom$ of $\CCC(\SA)$.}
		There exist several sign conventions regarding the differential of the complex obtained from applying the bifunctor $\Hom^\bullet(-,-)$ to complexes $M,N \in \CCC(\SA)$. The convention we will use is that, for $n \in \ZZ$, the differential
		\[
			d^n \colon \Hom^n(M,N)  \lto  \Hom^{n+1}(M,N),
		\]
		where $\Hom^n(M,N) = \prod_{j \in \ZZ} \Hom(M^j,N^{j+n})$, is defined as follows: if we let $f = (f^j)_{j \in \ZZ} \in \prod_{j \in \ZZ} \Hom(M^j,N^{j+n})$, then the composition of $d^n$ with the projection $p_j$ on the $j$-th factor is
		\[
			p_j \circ d^n(f) = f^{j+1} \circ d_M^j+ (-1)^{n+1} d_N^{j+n} \circ f^j.
		\]
		With this definition, one has that 
		\[
			\ker(\Hom^0(M,N) \overset{d^0}{\lto} \Hom^1(M,N)) = \Hom_{\CCC(\SA)}(M,N)	
		\]
		and 
		\[
			\Img(\Hom^{-1}(M,N) \overset{d^{-1}}{\lto} \Hom^0(M,N))
		\] 
		are just the nullhomotopic cochain complex morphisms, so 
		\[
			\h^0 \Hom^\bullet(M,N) = \Hom_{\K(\SA)}(M,N).
		\]
		Observe that we have chosen the sign to appear in the covariant part, although this is not always the case (\textit{cf.} \cite[p. 23]{yellow}). Then, our choice is closer to \cite[p. 62]{W} or \cite[p. 168]{HS}. The similarity comes, of course, after totalization, because in our double complexes differentials commute. This convention is helpful for two reasons. For one, it allows us to see $\Hom^\bullet(M,N)$ as the totalization by products of a suitable double complex; namely, the double complex $\Hom(M,N)^{\bullet\bullet}$ which in degree $(r,s)$ is equal to $\Hom(A^{-r},B^s)$ and whose differentials are
		\begin{equation*}
			\begin{aligned}
				&d_h^{rs} \coloneqq \Hom(d_M^{-r-1},N^s) \colon \Hom(M^{-r},N^s) \lto \Hom(M^{-r-1}, N^s), \text{ and} \\
				&d_v^{rs} \coloneqq (-1)^{s+1}\Hom(M^{-r},d_N^s) \colon \Hom(M^{-r},N^s) \lto \Hom(M^{-r}, N^{s+1}).
			\end{aligned}
		\end{equation*}
		With the differentials defined in this way, it becomes clear that $\Hom(M,N)^{\bullet\bullet}$ is a double complex, for the two differentials commute. Moreover, it follows from the definitions that
		\[ \Hom^\bullet(M,N) = \Tot^{\prod}(\Hom(M,N)^{\bullet\bullet}). \]
		The second reason is that we needed to fix a sign convention in order to prove Proposition \ref{hom-and-hocolim-holim} in a rigorous way. As a result, it will follow that homotopy limits and colimits are dual, as one would expect.
	\end{cosa}
	
	\begin{cosa}\label{higher-derived-limits}
		\textbf{Higher derived limits in $\mathsf{AB4^*}$ abelian categories.}
		Let $\SI$ be a small category and assume the abelian category $\SA$ satisfies $\mathsf{AB4^*}$---products are exact. By an inverse system in $\SA$ of shape $\SI$ we mean a functor $F \colon \SI^\op \to \SA$, that is, an object of $\cat(\SI^\op, \SA)$, the category of all contravariant functors from $\SI$ to $\SA$. Our goal is to compute the right derived functors of the left-exact functor
		\[ 
			\varprojlim \colon \cat(\SI^\op, \SA) \lto \SA,
		\]
		which makes sense because $\SA$ satisfies $\mathsf{AB3^*}$ and thus has all limits. Denote those right derived functors, to which we will refer as higher derived limits, as follows:
		\[ 
			\sideset{}{^{n}}\varprojlim \coloneqq \SR^n \varprojlim.
		\]
		As usual, these can be computed via injective resolutions provided the category $\cat(\SI^\op, \SA)$ has enough injectives. However, we will not assume that is the case. When $\SA$ satisfies $\mathsf{AB4^*}$, one can compute the higher derived limits as well via the construction of an explicit complex for each diagram $F \in \cat(\SI^\op, \SA)$, concentrated in non-negative degrees, whose $n$-th cohomology precisely yields $\sideset{}{^{n}}\varprojlim F$. The original idea goes back to Roos \cite{roos-1958}, whereas a more recent and general treatment can be found in \cite[A.3.]{Ntc}. 
		
		In this section we briefly recall the construction. For any inverse system $F$ of shape $\SI$, let us consider the nerve of the category $\SI$, $\CN_\bullet(\SI)$. That is, the simplicial set whose $k$-simplices, $\CN_k(\SI)$, are sequences of $k$ composable morphisms of $\SI$. An element of $\CN_k(\SI)$ is a chain of morphisms in $\SI$
		\[
			\textbf{i} \equiv i_0 \lto i_1 \lto \cdots \lto i_k,
		\]
		which we will denote by $\textbf{i}$ for short, and we will refer to the objects in the chain $\textbf{i} \in \CN_k(\SI)$ by $i_j$, for each $j \in \{0,\dots,k\}$. Additionally, given a chain $\textbf{i} \in \CN_{k+1}(\SI)$, we will denote by $\widehat{\textbf{i}}_j$ the chain in $\CN_{k}(\SI)$
		\[
			\widehat{\textbf{i}}_j \equiv i_0 \to \dots \to \widehat{i_j} \to \dots \to i_{k+1} 
		\] obtained after removing the object $i_j$.
		
		Using Neeman's notation, define the complex $N^\bullet(F)$ associated to the inverse system $F$ by letting the object in degree $k$ to be 
		\[
			N^k(F) \coloneqq \prod_{ \textbf{i} \in \CN_k(\SI)} F(i_0),
		\] where $N^k(F) = 0$ for all $k < 0$ and the differential $d^k \colon N^k(F) \to N^{k+1}(F)$ is determined by the following expression in each component:
		\begin{equation}\label{eq: differential-N-complex}
			p_{\textbf{i}}d^k = F(i_0 \to i_1)p_{\, \widehat{\textbf{i}}_0} + \sum_{j=1}^{k+1} (-1)^j \id_{F(i_0)}p_{\, \widehat{\textbf{i}}_j}.
		\end{equation}
		This definition is somewhat compressed, so let us spell it out in greater detail to make it clearer. The idea is that the morphisms involved in the definition are the obvious ones that are induced by deleting an object in the chain of morphisms $\textbf{i} \equiv i_0 \to i_1 \to \dots \to i_{k+1} \in \CN_{k+1}(\SI)$. Deleting the object $i_0$ gives the commutative diagram 
		\[
			\begin{tikzpicture}
				\matrix (m) [matrix of math nodes,
				row sep=3em, column sep=5em,
				text height=2ex,
				text depth=0.25ex]{
					\prod_{ \textbf{i} \in \CN_k(\SI)} F(i_0) &  \prod_{ \textbf{i} \in \CN_{k+1}(\SI)} F(i_0)  \\
					F(i_1) & F(i_0),  \\
				};
				\path[->, font=\scriptsize]
				(m-1-1) edge node[above] {$d^k$} (m-1-2)
				edge node[left] {$p_{\,\widehat{\textbf{i}}_0}$} (m-2-1)
				(m-1-2) edge node[right] {$p_{\,\textbf{i}}$} (m-2-2)
				(m-2-1) edge node[above] {$F(i_0 \to i_1)$} (m-2-2);
			\end{tikzpicture}
		\] whereas deleting any object $i_j$, with $1 \leq j \leq k+1$, corresponds to the commutative diagram
		\[
			\begin{tikzpicture}
				\matrix (m) [matrix of math nodes,
				row sep=3em, column sep=5em,
				text height=2ex,
				text depth=0.25ex]{
					\prod_{ \textbf{i} \in \CN_k(\SI)} F(i_0) &  \prod_{ \textbf{i} \in \CN_{k+1}(\SI)} F(i_0)  \\
					F(i_0) & F(i_0).  \\
				};
				\path[->, font=\scriptsize]
				(m-1-1) edge node[above] {$d^k$} (m-1-2)
				edge node[left] {$p_{\widehat{\textbf{i}}_j}$} (m-2-1)
				(m-1-2) edge node[right] {$p_{\textbf{i}}$} (m-2-2)
				(m-2-1) edge node[above] {$\id_{F(i_0)}$} (m-2-2);
			\end{tikzpicture}
		\]
		Note that for the construction of $N^\bullet(F)$ we used $\mathsf{AB3^*}$. The precise assertion from before is that, when $\SA$ moreover satisfies $\mathsf{AB4^*}$, the complex $N^\bullet(F)$ computes the higher derived limits of $F$, that is, $H^k(N^\bullet(F)) = \sideset{}{^{k}}\varprojlim F$ for all $k \geq 0$. This is shown in \cite[Lemma A.3.2.]{Ntc}, and the proof consists on showing that the collection of functors $\left\{T^k \coloneqq H^k(N^\bullet(-)), \, k \in \NN\right\}$, form a universal $\delta$-functor and that there is a natural isomorphism $T^0 \iso \varprojlim F$.
		
		\begin{rem}
			When dealing with the nerve of the small category $\SI$, we are considering any chain of composable morphisms---including degenerate chains, \emph{i.e.}, those containing identity morphisms. However, if we assume the small category $\SI$ is a partially ordered set, it is more convenient to do away with those and consider only non-degenerate chains of morphisms. 
			
			Denote by $\CN_\bullet'(\SI)$ the \emph{reduced nerve} of $\SI$, that is, the simplicial set whose $k$-simplices are the composable non-degenerate chains of morphisms of $\SI$. Analogously, for any $F \in \cat(\SI^\op, \SA)$, we will call $N'^\bullet(F)$ the complex defined in the same way as $N^\bullet(F)$ but replacing $\CN_\bullet(\SI)$ by the reduced nerve $\CN_\bullet'(\SI)$. Note that it is indeed well defined (and thus a complex) when $\SI$ is a poset, for in that case deleting an object in a non-degenerate chain of composable morphisms yields another non-degenerate chain, but that might not be the case for a general small category. The fact that $\h^0(N'^\bullet(F)) \cong \varprojlim F$ is still obvious---after all, the degenerate $1$-simplices are just identity morphisms and do not add any information to the computation of $\h^0$. Further, the collection of functors $\left\{T'^k, k \in \NN \right\}$, where $T'^k \coloneqq \h^k(N'^\bullet(-))$, constitutes an universal $\delta$-functor. This was already shown in \cite[Proposition 3]{roos-1958}, at least when $\SA$ is a concrete category (that is, when $\SA$ is equipped with a faithful functor $\SA \to \set$) and hence it makes sense to consider surjective morphisms. The idea is that the long exact sequence comes again from the fact that $\SA$ satisfies $\mathsf{AB4^*}$, and so it is enough to show that it is effaceable. This comes from two facts: one is the standard observation that an inverse system admits a monomorphism into a \textit{flabby} inverse system (see \ref{flabby-systems}); the second is that for any flabby inverse system $F \in \cat(\SI^\op, \SA)$ the cohomology $\h^k(N'^\bullet(F))$ vanishes for $k \geq 1$ (see \cite[Proposition 1]{roos-1958}, where the subsets of $\SI$ should be taken to be open in the Alexandrov topology described below). 
		\end{rem} 

	\end{cosa}
	
	\begin{cosa}\label{flabby-systems}
		\textbf{Inverse systems as sheaves, flabby systems and the Mittag-Leffler condition.}
		Given a poset $\SI$, we endow it with a topology, the Alexandrov topology, by letting the open sets be the lower sets. Explicitly, $|\SI|$ will be the topological space with the same underlying set as $\SI$, for which we declare a subset $U$ of $|\SI|$ to be open if and only if it is downward closed:
		\[ 
			\{j \in U \text{ and } i \leq j\} \implies i \in U. 
		\]
		Thus, the basic open subsets are the sets $|\SI|_{\leq j} \coloneqq \{i \in |\SI| \mid i \leq j\}$, with $j \in |\SI|$.
		
		Let $F \colon \SI^\op \to \SA$ be an inverse system and $\SA$ a complete abelian category (i.e. it satisfies $\mathsf{AB3^*}$). From $F$ we can construct a presheaf on the topological space $|\SI|$ that we write as $\tilde{F}$ by letting $\Gamma(U, \tilde{F}) \coloneqq \varprojlim F\vert_U$ for each open subset $U \subset |\SI|$, where $F\vert_U$ is the restriction of $F$ to the subcategory $U^\op$. With this definition, $\tilde{F}$ is automatically a sheaf on $|\SI|$, and its stalks at any $j \in |\SI|$ are just $\tilde{F}_{j} = F(j)$ (which also coincides with the sections along the basic open subset $|\SI|_{\leq j}$). Going in the other direction is straightforward: if we are given a sheaf $G$ on the topological space $|\SI|$, we can just define an inverse system $F \colon \SI^\op \to \SA$ by declaring $F(j)= G_{j}$, that is, the stalk of $G$ at $j$, which is exactly $\Gamma(|\SI|_{\leq j},G)$. This construction obviously establishes an (exact) isomorphism of categories between $\cat(\SI^\op, \SA)$ and the category of sheaves on the topological space $|\SI|$ (endowed with the Alexandrov topology) with values in $\SA$. As a consequence, when $\SA$ satisfies $\mathsf{AB3^*}$ the higher derived limits of $\varprojlim \colon \cat(\SI^\op, \SA) \to \SA$ are identified with the sheaf cohomology functors. 
		
		With this indentification in mind and when the category $\SA$ is concrete, it makes sense to consider the notion of a \textit{flabby} inverse system $F \colon \SI^\op \to \SA$, which is just asking the corresponding sheaf on $|\SI|$ to be flabby. That is, for any open subset $U$ of $|\SI|$, we ask the canonical map $\varprojlim F \to \varprojlim F\vert_U$, which corresponds to the restriction map $\Gamma(|\SI|, \tilde{F}) \to \Gamma(U, \tilde{F})$, to be surjective. 
	
		From the above discussion, it immediately follows that the higher derived limits vanish for flabby systems, since flabby sheaves are acyclic in sheaf cohomology.
		
		There exists a more general notion, that of a \textit{weakly flabby} inverse system \cite[p. 6]{jensen}, for which the conclusions are the same. We say an inverse system $F \colon \SI^\op \to \SA$ is \textit{weakly flabby} if for any \textit{filtered} open subset $U$ of $|\SI|$ the canonical map $\varprojlim F \to \varprojlim F\vert_U$ is surjective. For our purposes, it suffices to assume $\SA = R\md$ for a general ring $R$ for this notion (see Lemma \ref{lem: flabby-over-subsets-with-maximum} and Theorem \ref{thm:quis-abelian-cat-ab4*-gen}). 
		
		\begin{thm}[{\cite[Théorème 1.8.]{jensen}}]\label{thm: jensen-weakly-flabby}
			Let $\SA = R\md$ for a ring $R$ and $\SI$ a filtered poset. If an inverse system $F \colon \SI^\op \to \SA$ is weakly flabby, $\sideset{}{^{n}}\varprojlim F = 0$ for all $n \geq 1$.
		\end{thm}
		
		An even weaker notion than these is the Mittag-Leffler condition. Although it can be defined more generally, our interest lies in the following version: the inverse system $F \colon \SI^\op \to \SA$ is said to satisfy the \textit{epimorphic Mittag-Leffler condition} if for any $i, j \in \SI$ with $i \leq j$ the morphism $F(j) \to F(i)$ is an epimorphism. With the above terminology, it is the same as saying that the morphisms between stalks (equivalently, between basic open subsets) in the associated sheaf $\tilde{F}$ are epimorphisms. 
		
		There is a strengthening of the epimorphic Mittag-Leffler condition in the case of sequences. That is, suppose $\SI = \SI(\gamma)$ is the set of ordinals smaller than $\gamma$, with $\gamma$ an infinite ordinal. Just this once, we will denote with letters $\alpha, \beta$ the elements of $\SI$. An inverse system $F \colon \SI^\op \to \SA$ will be called just a sequence in $\SA$ of length $\gamma$. Note that the usual case, where $\SI = \NN$, is just letting $\SI = \SI(\omega)$. We declare a sequence $F \colon \SI^\op \to \SA$ to be \emph{strong epimorphic Mittag-Leffler} (see \cite[Definition A.3.10.]{Ntc}, where this notion is just called Mittag-Leffler sequence) if the following two conditions are met:
		\begin{enumerate}
			\item[(\romannumeral 1)] For any pair of ordinals $\alpha \leq \beta$ in $\SI$, the map $F(\beta) \to F(\alpha)$ is an epimorphism.
			\item[(\romannumeral 2)] For any limit ordinal $\beta \in \SI$, the map $F(\beta) \to \varprojlim_{\alpha < \beta} F(\alpha)$ is an epimorphism.
		\end{enumerate}
		
		Condition (\romannumeral 1) says the system is epimorphic Mittag-Leffler; condition (\romannumeral 2) tells us that in the associated sheaf $\tilde{F}$, the restriction morphism 
		\[
			\Gamma(|\SI|_{\leq \beta}, \tilde{F}) \lto \Gamma(\bigcup_{\alpha < \beta} |\SI|_{\leq \alpha}, \tilde{F})
		\]
		is also an epimorphism for any limit ordinal $\beta$. Obviously, if $\SI = \NN$, this definition is just the usual notion of epimorphic Mittag-Leffler sequence.
		
		Note that when $\SI = \SI(\gamma)$, the only non-basic open subsets are determined by the limit ordinals---given a limit ordinal $\beta \in \SI$, the subset $V$ of elements strictly smaller than $\beta$ is open, because $V = \bigcup_{\alpha < \beta} |\SI|_{\leq \alpha}$, but it is not basic since $\beta$ has no predecessor. Thus, asking for a sequence to be strong epimorphic Mittag-Leffler amounts to asking some extra restriction maps to be epi, so this notion is close to that of flabbiness. When the poset is a sequence in the usual sense, \emph{i.e.}, $\SI = \NN$, every open subset in the Alexandrov topology is clearly basic; therefore, it follows that when the category $\SA$ is concrete and such that epimorphisms are the surjective maps, the epimorphic Mittag-Leffler condition and flabbiness are equivalent. 
		
		However, it should be remarked that, generally speaking (and when it makes sense to compare them), the (strong) epimorphic Mittag-Leffler condition is weaker than flabbiness, for there might exist open subsets in the Alexandrov topology for which the restriction maps are not surjective. Therefore, the vanishing of higher derived limits for general (strong) epimorphic Mittag-Leffler systems does not come for free and it is often difficult to obtain.	
	\end{cosa}
	
	\begin{cosa}\label{countable-higher-derived-limits-abelian-cat}
		\textbf{Vanishing of higher derived limits of sequences.}
		In Sections \ref{Section:homotopy-limit} and \ref{Section:homotopy-limits-and-colocalizing-subcategories} we will be interested in (strong) epimorphic Mittag-Leffler sequences and the vanishing of their higher derived limits---specifically, in the vanishing of $\sideset{}{^{1}}\varprojlim$. This fact will be used in the proof of Lemma \ref{lem:colocalzing-mittag-leffler-countable-system} (and in the Remark before it).
		
		In the case of countable sequences, under certain conditions, it is clear why we only need $\sideset{}{^{1}}\varprojlim$ to vanish. Indeed, if $\SA$ satisfies $\mathsf{AB4^*}$, then the higher derived limits of a functor 
		\[
			F \colon \NN^\op \lto \SA
		\]
		can be computed as the cohomology of the so called Milnor two-term complex
		\[\prod_{i=0}^{\infty} F(i) \xrightarrow{1 - \text{shift}} \prod_{i=0}^{\infty} F(i),\]
		instead of the complexes $N^\bullet(F)$ or $N'^\bullet(F)$ of \ref{higher-derived-limits} (see \cite[Remark A.3.6.]{Ntc}). Note that the $\text{shift}$ map is defined as the composite
		\[
			\prod_{i=0}^{\infty} F(i) \xrightarrow{\pi} \prod_{i=1}^{\infty} F(i) \xrightarrow{\prod_{i=1}^{\infty} [F(i) \to F(i-1)]} \prod_{i=0}^{\infty} F(i),
		\]
		where $\pi$ is just the projection to the subproduct.
		Consequently, we have that $\varprojlim F = \ker(1 - \text{shift})$, $\sideset{}{^{1}}\varprojlim F = \cok(1- \text{shift})$ and $\sideset{}{^{n}}\varprojlim F = 0$ for $n \geq 2$. Then, one must only find conditions on $F$ so that $\sideset{}{^{1}}\varprojlim F$ vanishes. 
		
		\begin{rem}
			If the category $\SA$ is concrete and epimorphisms are surjective maps---like $\SA = R\md$ for a general ring $R$---we observed in \ref{flabby-systems} that a functor $F \colon \NN^\op \to \SA$ satisfying the epimorphic Mittag-Leffler condition is flabby, and thus $\sideset{}{^{1}}\varprojlim F$ also vanishes. The previous argument does not hold in general abelian categories, even those that satisfy $\mathsf{AB4^*}$ \cite[Proposition 3.12.]{Neeman2002}.
		\end{rem}
		
		Despite the last remark, the next result, stated for future reference and proven by Roos and Gabber, tells us that the epimorphic Mittag-Leffler condition is indeed sufficient as long as our abelian categories are nice enough.
		
		\begin{thm}[{\cite[Theorem 3.1.]{roos-2006}}]\label{thm:vanishing-countable-mittag-leffler-higher-limits}
		Let $\SA$ be an abelian category satisfying $\mathsf{AB3}$, $\mathsf{AB4^*}$ and \emph{having a generator}. Let $F \colon \NN^\op \to \SA$ be an inverse system satisfying the epimorphic Mittag-Leffler condition, that is, the maps $F(n+1) \to F(n)$, for any $n \in \NN$, are all epimorphisms. Then
		\[
			\sideset{}{^{i}}\varprojlim F = 0
		\]
		for $i \geq 1$.	
		\end{thm}
		
		From \ref{flabby-systems} we know that for more general sequences inm $\SA$ of the form $\SI = \SI(\gamma)$ for an ordinal $\gamma$, the strong epimorphic Mittag-Leffler condition is not far from flabbiness. Of course, the category $\SA$ must satisfy certain conditions for both notions to be comparable. It has also been observed that if $\gamma = \omega$ (\textit{i.e.}, $\gamma$ is the smallest infinite ordinal), both notions agree.
		
		One might wonder if for certain categories $\SA$ this also happens for any infinite ordinal $\gamma$. After all, for the associated sheaf $\tilde{F}$, the only restriction morphisms which need to be shown surjective are those where the bigger open subset is non-basic (for a much more detailed discussion, consult \cite[Remark A.3.12.]{Ntc}). Neeman showed that this holds when $\SA = \ab$, so it also holds true when $\SA = R\md$ for a general ring $R$, using an analogous argument:
		
		\begin{lem}[{\cite[Lemma A.3.13 and Corollary A.3.14]{Ntc}}]\label{lem: sequences-flabby-neeman}
			Let $\SA = R\md$ for a ring $R$, $\SI = \SI(\gamma)$ for an infinite ordinal $\gamma$, and let $F \colon \SI^\op \to \SA$ be an inverse system satisfying the strong epimorphic Mittag-Leffler condition. Then the associated sheaf $\tilde{F}$ on $|\SI|$ is flabby and, as a consequence, $\sideset{}{^{n}}\varprojlim F = 0$ for $n \geq 1$.
		\end{lem}
		
		\begin{rem}
			This result will only be used in the Remark after Theorem \ref{thm:quis-abelian-cat-ab4*-gen} as an example of its application.
		\end{rem}
		
	\end{cosa}

\section{Homotopy limits of complexes}\label{Section:homotopy-limit}
	In this section we will assume the abelian category $\SA$ satisfies $\mathsf{AB3^*}$, that is, small products (and thus small limits) exist. With this in mind, we introduce the homotopy limit of a diagram of complexes in $\CCC(\SA)$. Its construction is dual to that of the homotopy colimit as it was presented in \cite[\S 2]{AJS}. 
	
	The diagrams we consider have the shape of a \textit{filtered} poset. In constrast with the previous section, we will denote it by $\Lambda$, and its elements by $s \in \Lambda$. Therefore, the notation $\textbf{s} \in \CN'_k(\Lambda)$ refers to a non-degenerate chain of morphisms of $\Lambda$ 
	\[
		\textbf{s} \equiv s_0 \to \dots \to s_k. 
	\]
	We consider functors $F \in \cat(\Lambda^\op, \CCC(\SA))$, that is, cofiltered inverse systems of shape $\Lambda$ in $\CCC(\SA)$.
	
	With the notation introduced in \ref{higher-derived-limits}, we construct a double complex associated to $F$, which we denote by $\left(B^{\prod}(F),d_h,d_v\right)$. Define, for $j, k \in \ZZ$:
	\begin{equation*}
		\begin{aligned}
			B^{\prod}(F)^{kj} \coloneqq \begin{cases}
				0 &\text{ if } k < 0 \\
					\prod_{ \textbf{s} \in \CN'_k(\Lambda)} F(s_0)^j &\text{ if } k \geq 0,
				\end{cases}
		\end{aligned}
	\end{equation*}
	where $s_0$ must henceforth be understood as the first object in the chain $\textbf{s}$ indicating the corresponding factor in the product. Observe that we are only considering the non-degenerate chains of morphisms of $\Lambda$, \textit{i.e.}, we consider the reduced nerve of $\Lambda$. Clearly, this notation allows us to make a more compact definition than the one given in \cite{AJS} of the analog double complex, which we would denote in this setting by $B^{\coprod}(F)$. The vertical differential $d_v$ is the one induced by the complexes themselves. Namely, in horizontal degree $k$, it is the product over $\CN'_k(\Lambda)$ of the differentials of each complex. As for the horizontal differential 
	\[
		B^{\prod}(F)^{kj} = \prod_{ \textbf{s} \in \CN'_k(\Lambda)}F(s_0)^j \xrightarrow{\,d_h^{kj}\,} B^{\prod}(F)^{k+1j} = \prod_{ \textbf{s} \in \CN'_{k+1}(\Lambda)}F(s_0)^j,
	\] we define it using the universal property of the product, where we denote by $p_{\textbf{s}}$ the canonical projection indicated by $\textbf{s} \in \CN'_k(\Lambda)$
	\[
		\prod_{\textbf{s} \in \CN'_k(\Lambda)}F(s_0)^j \lto F(s_0)^j,
	\]
	completely analogously as the way the differential of the complex $N'^\bullet(F)$ was defined (see \ref{eq: differential-N-complex}). Explicitly, for $\textbf{s} \in \CN'_{k+1}(\Lambda)$,
	\[ p_{\textbf{s}}d_h^{kj} = F(s_0 \to s_1)^j p_{\,\widehat{\textbf{s}}_0} + \sum_{i=1}^{k+1} (-1)^i \id_{F(s_0)^j} p_{\,\widehat{\textbf{s}}_i}.\]
	
	The fact that $d_v$ is a differential is obvious, while proving the same for $d_h$ is a routinary but tedious check. One can also check that with these definitions $(B^{\prod}(F),d_h,d_v)$ becomes a double complex in the sense of Section \ref{Section:Conventions}. 
	
	As soon as we know $(B^{\prod}(F),d_h,d_v)$ is a double complex, we know we can totalize (as discussed in \ref{totalization}), and so we make the following definition. Given a functor $F \in \cat(\Lambda^\op, \CCC(\SA))$, we define the \textit{homotopy limit of $F$} to be 
	\[ \holim{}F \coloneqq \Tot^{\prod}\left(B^{\prod}(F)\right). \]
	
	\noindent
	The construction is clearly functorial, so, as a consequence, we have defined a functor
	\[\holim{} \!\!\! = \Tot^{\prod} \circ B^{\prod} \colon \cat(\Lambda^\op, \CCC(\SA)) \lto \CCC(\SA).\]
	
	The next lemma is an observation, perhaps obvious after the construction.
	
	\begin{lem}\label{complexes-concentrated-degree-0}
		Let $\SA$ and $\Lambda$ be as before and let $F \colon \Lambda^\op \to \SA$ be an inverse system, which can be thought as an inverse system $\Lambda^\op \to \CCC(\SA)$ (still denoted by $F$) via the identification of $\SA$ with the degree-$0$ complexes of $\CCC(\SA)$. If $\SA$ satisfies $\mathsf{AB4^*}$, then
		\[ \h^n(\!\!\!\holim{} \! F) \cong \sideset{}{^{n}}\varprojlim F.\]
	\end{lem}
	\begin{proof}
		The complexes are concentrated in degree $0$, so $B^{\prod}(F)$ only has information in vertical degree $0$. The horizontal differential $d_h$ is identical to that of the complex $N'^\bullet(F)$ discussed in \ref{higher-derived-limits}, and after totalization, since there are no nonzero vertical differentials in any horizontal degree, we get
		\[\holim{}F = N'^\bullet(F).\]
		The result follows, since $\h^n(N'^\bullet(F)) \cong \sideset{}{^{n}}\varprojlim F$ holds (see the Remark in \ref{higher-derived-limits}).
	\end{proof}
	
	The objective of the remainder of this section is exploring the basic properties that the homotopy limit satisfies. 
	The main task is exploring the relationship between the limit and the homotopy limit. Let us begin with the next proposition:
	
	
	\begin{prop}\label{prop:homotopy-equiv-finite-system}
		Let $\Lambda$ be a finite poset with maximum $\mu$, $\SA$ an abelian category satisfying $\mathsf{AB3^*}$ and $F \colon \Lambda^\op \to \CCC(\SA)$ be any diagram. Then, $F(\mu) = \varprojlim F$ is homotopically equivalent to $\!\!\!\holim{} F$.
	\end{prop}
	\begin{proof}
		With $\mu = \max\{s \mid s \in \Lambda\}$, we define two maps of double complexes. First, by thinking of $F(\mu)^\bullet$ as a double complex concentrated in horizontal degree $0$, 
		\begin{equation*}\label{eq: map1-finite-poset}
			B^{\prod}(F)^{\bullet\bullet} \lto F(\mu)^{0\bullet}
		\end{equation*}
		by setting $B^{\prod}(F)^{0j} = \prod_{s \in \Lambda} F(s)^j \to F(\mu)^j$ to be the projection and, of course, the morphisms on the rest of horizontal degrees to be the zero morphism. It is trivially a morphism of double complexes.
		
		On the other hand, define
		\begin{equation}\label{eq: map2-finite-poset}
			F(\mu)^{0\bullet} \lto B^{\prod}(F)^{\bullet\bullet}
		\end{equation}
		by letting the map in horizontal deree $0$, for each $j \in \ZZ$, be defined by the commutative diagrams
		\[
			\begin{tikzpicture}
				\matrix (m) [matrix of math nodes,
				row sep=3em, column sep=3em,
				text height=2ex,
				text depth=0.25ex]{
					F(\mu)^j &  \prod_{ s \in \Lambda} F(s)^j = B^{\prod}(F)^{0j}  \\
					 & F(s)^j.  \\
				};
				\path[->, font=\scriptsize]
				(m-1-1) edge node[auto] {} (m-1-2)
				(m-1-1) edge node[below, sloped] {$F(s \to \mu)^j$} (m-2-2)
				(m-1-2) edge node[right] {} (m-2-2);
			\end{tikzpicture}
		\]
		The reader can check that (\ref{eq: map2-finite-poset}) is a morphism of double complexes.
		
		If we totalize these two morphisms by products, we obtain
		\[ 
			\varphi \colon \!\!\!\holim{}F \lto \Tot^{\prod}\left(F(\mu)^{0\bullet}\right) = F(\mu)
		\]
		for the first one, and 
		\[
			\psi \colon \Tot^{\prod}\left(F(\mu)^{0\bullet}\right) = F(\mu) \lto \holim{} F
		\]
		for the second one. 
		
		Notice that $\varphi \psi = \id_{F(\mu)}$ because composing the two morphisms of double complexes in this order gives the identity in $F(\mu)$ (by definition of the second map of double complexes) and totalizing is functorial. On the other hand, we must show that $\psi \varphi$ is homotopically equivalent to the identity. To accomplish this, we will define a homotopy 
		\[
			\sigma = \left(\sigma^n \colon (\!\!\!\holim{} F)^n \lto (\!\!\!\holim{}F)^{n-1}\right)_{n \in \ZZ} 
		\]
		and check that
		\begin{equation}\label{eq: homotopy-equiv-finite-subsets-maximum}
			(\psi \varphi - 1)^n = d^{n-1} \sigma^n + \sigma^{n+1} d^n
		\end{equation}
		for each $n \in \ZZ$, where by $d$ we mean the differential of the homotopy limit of $F$. Since at each degree $n \in \ZZ$ we have
		\[
			(\!\!\!\holim{}F)^n = \prod_{k+j=n} \,\, \prod_{ \textbf{s} \in \CN'_k(\Lambda)} F(s_0)^j,
		\]
		we will denote the projections as shown below
		\[
			\prod_{k+j=n} \,\, \prod_{ \textbf{s} \in \CN'_k(\Lambda)} F(s_0)^j \overset{p_j}{\lto} \prod_{ \textbf{s} \in \CN'_k(\Lambda)} F(s_0)^j \overset{p_{\textbf{s}}}{\lto} F(s_0)^j.
		\]
		and we will actually check the equality (\ref{eq: homotopy-equiv-finite-subsets-maximum}) above holds after composing with these two projections in degree $n$. Observe that after being given an $n \in \ZZ$ and fixing $j \leq n$,  by applying $p_j$ we obtain that $k=n-j$. We will use this henceforth without mention. Construct the homotopy by setting the morphism $\sigma^n$ from degree $n$ to degree $n-1$, to be defined as follows. Let $j \leq n-1$ and $\textbf{s} \in \CN'_{n-j-1}(\Lambda)$ be a chain of morphisms, and define both $\zeta^j$ and $\sigma^n$ by the universal property of the product, as the diagram below dictates:
		\[
			\begin{tikzpicture}
				\matrix (m) [matrix of math nodes,
				row sep=3em, column sep=5em,
				text height=2ex,
				text depth=0.25ex]{
					\prod_{i \leq n} \prod_{ \textbf{t} \in \CN'_{n-j}(\Lambda)} F(t_0)^j &  \prod_{j \leq n-1} \prod_{ \textbf{s} \in \CN'_{n-j-1}(\Lambda)} F(s_0)^j  \\
					\prod_{ \textbf{t} \in \CN'_{n-j}(\Lambda)} F(t_0)^j & \prod_{ \textbf{s} \in \CN'_{n-j-1}(\Lambda)} F(s_0)^j  \\
					& F(s_0)^j \\
				};
				\path[->, font=\scriptsize]
				(m-1-1) edge node[auto] {$\sigma^n$} (m-1-2)
				(m-1-1) edge node[left] {$p_j$} (m-2-1)
				(m-1-2) edge node[right] {$p_j$} (m-2-2)
				(m-2-1) edge node[auto] {$\zeta^j$} (m-2-2)
				(m-2-1) edge node[below] {$f_{\textbf{s}}^j$} (m-3-2)
				(m-2-2) edge node[right] {$p_{\textbf{s}}$} (m-3-2);
			\end{tikzpicture}
		\]
		where the subscript $\textbf{s}$ in $f_{\textbf{s}}$ is determined by the chain of maps $\textbf{s} \in \CN'_{n-j-1}(\Lambda)$ corresponding to the factor onto which we are projecting (thus completely determining $\zeta^j$) and it is such that
		\begin{equation}\label{eq: maps-defining-homotopy}
		 f_{\textbf{s}}^j = \begin{cases}
		 	0 &\text{ if } s_{n-j-1} = \mu, \\
		 	(-1)^{k-1} \id_{F(s_0)^j}p_{\textbf{s} \to \mu} &\text{ if } s_{n-j-1} \neq \mu.
		 \end{cases}
		\end{equation}
		By $\textbf{s} \to \mu$ we mean the non-degenerate chain belonging to $\CN'_{n-j}(\Lambda)$ obtained after adjoining the morphism $s_{n-j-1} \to \mu$ to the chain $\textbf{s} \in \CN'_{n-j-1}(\Lambda)$.
		
		To check the promised equality (\ref{eq: homotopy-equiv-finite-subsets-maximum}), first we observe that for $\textbf{s} \in \CN'_{n-j}(\Lambda)$, 
		\small
		\begin{equation}\label{cases-to-compare}
			p_{\textbf{s}} p_j (\psi \varphi - 1)^n = \begin{cases}
				-p_{\textbf{s}} p_j &\text{ if } j\neq n \\
				0 &\text{ if } j=n \text{ and } s_{n-j}=s_0= \mu \\
				F(s_0 \to \mu)^j p_{\mu} p_j - p_{s_0}p_j  &\text{ if } j=n \text{ and } s_{n-j}=s_0\neq \mu.
			\end{cases}
		\end{equation} 
		\normalsize
		The case $j\neq n$ follows from the fact that, in that case, $n-j \neq 0$ and then both $\varphi$ and $\psi$ are zero after postcomposing with $p_j$. Similarly, the other cases (which we only differentiate because in one case we obtain $0$) follow easily from the definition of $\varphi$ and $\psi$.
		
		Now let us show that the other side (\textit{i.e.}, the right hand side) of the equality (\ref{eq: homotopy-equiv-finite-subsets-maximum}) agrees with (\ref{cases-to-compare}): fixing $j \in \ZZ$ and $\textbf{s} \in \CN'_{n-j}(\Lambda)$, composition with the projections gives
		\[
			p_{\textbf{s}} p_j \left(d^{n-1}\sigma^n + \sigma^{n+1}d^n\right) = p_{\textbf{s}} p_j d^{n-1}\sigma^n + p_{\textbf{s}} p_j \sigma^{n+1}d^n,
		\] 
		and so we inspect each summand separately.
		
		On the first one, from the definition of $\sigma$ above and the description of the differential of the totalization by products given in (\ref{differential-totalization-products}), we obtain
		\small
		\begin{equation}\label{summand-one}
			\begin{aligned}
				p_{\textbf{s}} p_j d^{n-1}\sigma^n &=\\
				&= p_{\textbf{s}} p_j \left(\sum_{j \in \ZZ} d_h^{n-j-1 \, j} + (-1)^{n-j}d_v^{n-j \, j-1}\right) \sigma^n\\
				&=\left(F(s_0 \to s_1)^j p_{\widehat{\textbf{s}}_0} p_j + \sum_{i=1}^{n-j} (-1)^i \id_{F(s_0)^j}p_{\widehat{\textbf{s}}_i} p_j + (-1)^{n-j} d_{F(s_0)}^{j-1} p_{\textbf{s}}p_{j-1}\right) \sigma^n \\
				&= F(s_0 \to s_1)^j f_{\widehat{\textbf{s}}_0}^j p_j + \sum_{i=1}^{n-j} (-1)^i \id_{F(s_0)^j} f_{\widehat{\textbf{s}}_i}^j p_j + (-1)^{n-j} d_{F(s_0)}^{j-1} f_{\textbf{s}}^{j-1} p_{j-1},
			\end{aligned}
		\end{equation} 
		\normalsize
		where (\ref{differential-totalization-products}) is used in the second equality.
		
		As for the second summand, again by the definitions and (\ref{differential-totalization-products}) we can write: 
		\begin{equation}\label{summand-two}
			\begin{aligned}
				 p_{\textbf{s}} p_j \sigma^{n+1}d^n = f_{\textbf{s}}^j p_j d^n &= f_{\textbf{s}}^j \left(d_h^{n-j j} p_j + (-1)^{n-j+1}d_v^{n-j+1 j-1} p_{j-1}\right) \\
				 &=  f_{\textbf{s}}^j d_h^{n-j j} p_j + (-1)^{n-j+1} f_{\textbf{s}}^j d_v^{n-j+1 j-1} p_{j-1}.
			\end{aligned}
		\end{equation}
		Depending on whether $f_{\textbf{s}}^j$ is zero or not (an thus can be factored as shown before, see (\ref{eq: maps-defining-homotopy})), we can develop this expression further. It will be done below, when discussing cases.
		
		Combining the expressions we have obtained for the two summands, we distinguish the same cases as we did in (\ref{cases-to-compare}) to check the equality
		\[
			p_{\textbf{s}}p_j (\psi \varphi - 1)^n = 	p_{\textbf{s}} p_j \left(d^{n-1}\sigma^n + \sigma^{n+1}d^n\right).
		\]
		
		\textbf{Case $j \neq n$}. We distinguish two subcases, whether $s_{n-j} = \mu$ or $s_{n-j}\neq \mu$:
		
		\textbf{(a)} If $s_{n-j} = \mu$, by definition of $f_{\textbf{s}}$, the expression corresponding to the first summand (\ref{summand-one}) vanishes except for the last term in $\sum_{i=1}^{n-j} (-1)^i \id_{F(s_0)^j} f_{\widehat{\textbf{s}}_i}^j p_j$, that is,
		\[
			(-1)^{n-j} \id_{F(s_0)^j} f_{\widehat{\textbf{s}}_{n-j}}^j p_j.
		\]
		In the expression corresponding to the second summand (\ref{summand-two}), everything vanishes for the same reason (the definition of $f_{\textbf{s}}$, since $s_{n-j} = \mu$). Therefore, all that remains is the expression above, and since $s_{n-j-1} \neq \mu$, the map $f_{\widehat{\textbf{s}}_{n-j}}^j$ factors, and we get
		\begin{equation*}
			\begin{aligned}
				&(-1)^{n-j} \id_{F(s_0)^j} f_{\widehat{\textbf{s}}_{n-j}}^j p_j = (-1)^{n-j} (-1)^{n-j-1} \id_{F(s_0)^j} p_{\textbf{s}} p_j = -\id_{F(s_0)^j} p_{\textbf{s}} p_j.
			\end{aligned}		
		\end{equation*}
		
		\textbf{(b)} Let us assume $s_{n-j} \neq \mu$. In this case, all summands survive in both (\ref{summand-one}) and (\ref{summand-two}), and thus we can factor the instances of $f_{\textbf{s}}^j$. As a result, we obtain the following expression combining both summands: 
		
		\begin{equation*}
			\begin{aligned}
				(\ref{summand-one}) + (\ref{summand-two}) =& (-1)^{n-j-1} F(s_0 \to s_1)^j p_{\widehat{\textbf{s}}_{0} \to \mu} p_j \\
				&+ \sum_{i=1}^{n-j} (-1)^{i+n-j-1} \id_{F(s_0)^j} p_{\widehat{\textbf{s}}_{i} \to \mu} p_j  +  (-1)^{2(n-j)} d_{F(s_0)}^{j-1}  p_{\textbf{s} \to \mu} p_{j-1}\\
				&+ (-1)^{n-j} F(s_0 \to s_1)^j p_{\widehat{(\textbf{s} \to \mu)}_{0} } p_j + \sum_{i=1}^{n-j+1} (-1)^{i+n-j}  \id_{F(s_0)^j} p_{\widehat{(\textbf{s} \to \mu)}_{i}} p_j \\
				&+ (-1)^{2(n-j)+1} d_{F(s_0)}^{j-1} p_{\textbf{s} \to \mu} p_{j-1}  \\
				=& (-1)^{2(n-j)+1} \id_{F(s_0)^j} p_{\widehat{(\textbf{s} \to \mu)}_{n-j+1}} p_j = - \id_{F(s_0)^j} p_{\textbf{s}} p_j.
			\end{aligned}
		\end{equation*}
		Clearly, every term cancels with another except for the term corresponding to the index $i = n-j+1$ in the third line, as shown in the last two equalities. But this lets us conclude for this case, because it coincides with the expression we obtained in \textbf{(a)}, as well as with the expression in (\ref{cases-to-compare}).
		
		\textbf{Case $j=n$}. Although we distinguish two cases inside this one, it is useful to make a general observation first. Since $n-j=0$ (the chains of morphisms have length $0$ and therefore consist only of an object), observe that in the first summand (\ref{summand-one}) only the last term is nonzero, for the horizontal differential coming from degree $-1$ is zero, thus the first two terms (the second being a sum) vanish. Thus, the sum of (\ref{summand-one}) and (\ref{summand-two}) becomes
		\begin{equation}\label{case-j=n-expression}
			d_{F(s_0)}^{n-1} f_{s_0}^{n-1} p_{n-1} + f_{s_0}^{n} d_h^{0 n} p_n - f_{s_0}^j d_v^{1 n-1} p_{n-1}.
		\end{equation}
		
		\textbf{(a)} Let us consider the subcase $j=n$ and $s_{n-j}=s_0=\mu$. Just as in the first case we studied, $f_{s_0}=0$, so every term in (\ref{case-j=n-expression}) (and thus the whole expression) vanishes. This agrees with (\ref{cases-to-compare}), so we are done.
		
		\textbf{(b)} If $j=n$ and $s_{n-j}=s_0 \neq \mu$, then the morphisms $f_{s_0}^j$ do not vanish; so we factor them and the expression (\ref{case-j=n-expression}) becomes
		\small
		\begin{equation*}
			\begin{aligned}
				d_{F(s_0)}^{n-1}& \id_{F(s_0)^{n-1}} p_{s_0 \to \mu} p_{n-1} + \id_{F(s_0)^n} p_{s_0 \to \mu} d_h^{0n} p_n - \id_{F(s_0)^n} p_{s_0 \to \mu} d_v^{1 n-1} p_{n-1} \\
				=& d_{F(s_0)}^{n-1} \id_{F(s_0)^{n-1}} p_{s_0 \to \mu} p_{n-1} + F(s_0 \to \mu) p_{\mu} p_n - \id_{F_(s_0)^n} p_{s_0} p_n \\
				&- d_{F(s_0)}^{n-1} \id_{F(s_0)^{n-1}} p_{s_0 \to \mu} p_n = \\
				=& F(s_0 \to \mu) p_{\mu} p_n - \id_{F(s_0)^n} p_{s_0} p_n,
			\end{aligned}
		\end{equation*}
		\normalsize
		because after the first equality the first and last term cancel each other. The result agrees with that of (\ref{cases-to-compare}), and we are done.
	\end{proof}
	
	The previous Proposition exhibits the interest, for computing (homotopy) limits, of the finite subsets with maximum of a given poset $\Lambda$. After all, the limit of a diagram $F \colon \Lambda^\op \to \SC$, where $\SC$ is any complete category, is the limit of its subdiagrams with maximum (in $\Lambda$). Let us make this precise: define the set 
	\[
		\pfinmax(\Lambda) \coloneqq \{E \subset \Lambda \mid E \text{ is finite and has a maximum element}\}
	\] 
	and denote by $\mu_E = \max\{s \mid s \in E\}$ for any $E \in \pfinmax(\Lambda)$. Obviously, the relation $E' \subset E$ for any $E, E' \in \pfinmax(\Lambda)$ determines a partial order in $\pfinmax(\Lambda)$. Moreover, the fact that $\Lambda$ is filtered implies that so is $\pfinmax(\Lambda)$.
	The assertion is that, if we denote by $F\vert_E \colon E^\op \to \SC$, for any subset $E$ of $\Lambda$, the restriction of $F$ to the subcategory $E^\op$, then
	\begin{equation}\label{limit-finite-subsystems-maximum}
		\varprojlim_{s \in \Lambda} F(s) = \varprojlim_{E \in \pfinmax(\Lambda)} F(\mu_E) = \varprojlim_{E \in \pfinmax(\Lambda)} \varprojlim_{s \in E} F(s).
	\end{equation}
	The only completely non obvious statement is the first equality. The diagram $\pfinmax(\Lambda)^\op \to \SC$ given by $E \mapsto F(\mu_E)$ is just a ``bigger'' version of the diagram $F$, where for each $s \in \Lambda$ one usually has more than one $E \in \pfinmax(\Lambda)$ with $\mu_E=s$. But among those elements $E, E'$ of $\pfinmax(\Lambda)$ with $\mu_E = s = \mu_{E'}$, provided they are comparable ($E' \subset E$), the maps in the diagram $F(\mu_E) = F(s) \to F(s) = F(\mu_{E'})$ are the identity. For different, comparable elements $s, t \in \Lambda$ (say $t \leq s$) and subsets $E, E' \in \pfinmax(\Lambda)$ such that $\mu_E= s$, $\mu_{E'}=t$ and $E' \subset E$, the map  $F(\mu_E) = F(s) \to F(t) = F(\mu_{E'})$ is just $F(t \to s)$. Therefore, one easily sees that $\varprojlim_{s \in \Lambda} F(s)$ is the universal limit cone of the diagram $E \mapsto F(\mu_E)$ (or vice versa).
	
	In other words, there is a functor
	\[
		u \colon \pfinmax(\Lambda)^\op \lto \Lambda^\op
	\]
	mapping $E$ to its maximum $\mu_E$ and mapping maximums in the obvious way on morphisms in $\pfinmax(\Lambda)^\op$. The functor $u$ induces, by precomposition, a functor
	\[
		u^* \colon \cat(\Lambda^\op, \SC) \lto \cat(\pfinmax(\Lambda)^\op, \SC)
	\] taking $F \colon \Lambda^\op \to \SC$ into $u^* F = F \circ u \colon \pfinmax(\Lambda)^\op \to \SC$, which is precisely the inverse system for which we take limits in the middle term of (\ref{limit-finite-subsystems-maximum}). Therefore, the first equality of (\ref{limit-finite-subsystems-maximum}) (and thus the assertion from before) can be thought as the isomorphism between the functors
	\begin{equation}\label{limit-finite-subsystems-u*}
		\varprojlim_{\Lambda} = \varprojlim_{\pfinmax(\Lambda)} \circ u^*.
	\end{equation}
	The above isomorphism makes sense whenever the category $\SC$ is complete or, at least, if the purported limits exist.
	
	\begin{lem}\label{lemma:holim-as-lim}
		Let $\SA$ be an abelian category satisfying $\mathsf{AB3^*}$, $\Lambda$ a filtered poset, and $F \colon \Lambda^\op \to \CCC(\SA)$ any inverse system. If, for any $E \subset \Lambda$, we denote by $F\vert_E$ the restriction of $F$ to the subcategory $E^\op$ of $\Lambda^\op$, the homotopy limit of $F$ can be computed as the limit of the diagram $H \colon \pfinmax(\Lambda)^\op \to \CCC(\SA)$ defined by $E \mapsto \holim{}F\vert_E$, \textit{i.e.},
		\[ \holim{} F \cong \varprojlim_{E \in \pfinmax(\Lambda)} \!\!\! \holim{}F\vert_E.\]
	\end{lem}
	\begin{proof}
		Given $E, E' \in \pfinmax(\Lambda)$ such that $E' \subset E$, it is clear the reduced nerves satisfy $\CN'_\bullet(E') \subset \CN'_\bullet(E)$, so we can define the obvious morphism to a subproduct. That is, for any $k, j \in \ZZ$ and $\textbf{s} \in \CN'_k(E')$,
		\[
			\begin{tikzpicture}
				\matrix (m) [matrix of math nodes,
				row sep=3em, column sep=3em,
				text height=2ex,
				text depth=0.25ex]{
					B^{\prod}(F\vert_E)^{kj} = \prod_{ \textbf{t} \in \CN'_k(E)}F(s_0)^j &  \prod_{ \textbf{s} \in \CN'_k(E')}F(s_0)^j = B^{\prod}(F\vert_{E'})^{kj}   \\
					F(s_0)^j & F(s_0)^j.  \\
				};
				\path[->, font=\scriptsize]
				(m-1-1) edge node[auto] {$\zeta^{kj}$} (m-1-2)
				(m-1-1) edge node[left] {$p_{\textbf{s}}$} (m-2-1)
				(m-1-2) edge node[right] {$p_{\textbf{s}}$} (m-2-2)
				(m-2-1) edge node[auto] {$\id_{F(s_0)^j}$} (m-2-2);
			\end{tikzpicture}
		\]
		It obviously commutes with the vertical differentials, and it is a straightforward verification that it also commutes with the horizontal differentials, so this defines a double complex morphism. The category $\SA$ is in particular pointed, so the morphisms $\zeta^{kj}$ are in fact retractions. 
		
		After totalization by products, we obtain a morphism 
		\[
			\Tot^{\prod}\left(B^{\prod}(F\vert_E)\right) \lto \Tot^{\prod}\left(B^{\prod}(F\vert_{E'})\right),
		\]
		which again is a retraction at every degree $n \in \ZZ$, because a product of retractions is again a retraction. Therefore, we have defined the inverse system of the statement, $H \colon \pfinmax(\Lambda)^\op \to \CCC(\SA)$, which is given by
		\[
			E \mapsto \holim{}F\vert_E.
		\]
		
		Now, it is obvious that $\Tot^{\prod}(B^{\prod}(F))$, together with projections defined in an analogous way to the morphisms above, is a \textit{cone over the diagram $H$}\footnote{Not to be confused with the \textit{mapping cone} of a map of complexes, which we will just call \textit{cone} (see the proof of Theorem \ref{thm:quis-abelian-cat-ab4*-gen})}. We want to show it is universal. But that is easy---given a cone $C \in \CCC(\SA)$ over the diagram $H$, we want to show the existence of a unique dashed arrow making commutative the diagram
		\[
			\begin{tikzpicture}
				\matrix (m) [matrix of math nodes,
				row sep=1.6em, column sep=3em,
				text height=2ex,
				text depth=0.25ex]{
				C & & \Tot^{\prod}\left(B^{\prod}(F)\right)   \\
				 & \Tot^{\prod}\left(B^{\prod}(F\vert_E)\right) &  \\
				& \Tot^{\prod}\left(B^{\prod}(F\vert_{E'})\right). & \\
				};
				\path[->, font=\scriptsize]
				(m-1-3) edge node[auto] {} (m-2-2)
				(m-1-3) edge node[auto] {} (m-3-2)
				(m-2-2) edge node[auto] {} (m-3-2)
				(m-1-1) edge node[auto] {} (m-2-2)
				(m-1-1) edge node[auto] {} (m-3-2)
				(m-1-1) edge[dashed] node[above] {$\exists!$} (m-1-3);
			\end{tikzpicture}
		\]
		In each degree $n \in \ZZ$, an arrow $C^n \to \Tot^{\prod}\left(B^{\prod}(F)\right)^n$ is completely defined after postcomposition with the projections $p_j$ and $p_{\textbf{s}}$ such that $k+j=n$ (where $k$ is the length of $\textbf{s}$). In other words, it is completely determined by defining each morphism $C^n \to F(s_0)^j$, where $F(s_0)^j$ is indexed $j$ and a chain of morphisms $\textbf{s} \equiv s_0 \to \dots \to s_k$ and $k+j=n$. But for each pair $(j, \textbf{s})$, the chain $\textbf{s}$ has finite length, so there exists $E \in \pfinmax(\Lambda)$ with $\textbf{s} \in \CN'_k(E)$ (for instance, the set $\{s_0. \dots, s_k\} \subset \Lambda$ does the trick), and the arrow $C^n \to \Tot^{\prod}\left(B^{\prod}(F\vert_E)\right)^n$ determines the wanted morphism $C^n \to F(s_0)^j$. This is easily seen to be independent of the chosen $E \in \pfinmax(\Lambda)$ because of the commutativity of the above diagram, so we get a uniquely defined morphism $C \to \Tot^{\prod}\left(B^{\prod}(F)\right)$ and we are done.	
	\end{proof}

	\begin{lem}\label{lem: flabby-over-subsets-with-maximum}
		Let $\SA$ be a concrete, abelian category satisfying $\mathsf{AB3^*}$ and $\Lambda$ be a filtered poset. If an inverse system $F \colon \Lambda^\op \to \CCC(\SA)$ is degree-wise weakly flabby, so is the inverse system
		\[
			u^* F \colon \pfinmax(\Lambda)^\op \lto \CCC(\SA), \quad E \mapsto F \circ u(E) = F(\mu_E),
		\]
		where $\mu_E$ is the maximum element of $E \in \pfinmax(\Lambda)$.
	\end{lem}
	\begin{proof}
		As in \ref{flabby-systems}, denote by $\widetilde{u^*F}$ the sheaf on $|\pfinmax(\Lambda)|$ associated to $u^*F$ and let $U \subset |\pfinmax(\lambda)|$ be a filtered open subset. 
		
		From $U$ we can obtain two subsets of $|\Lambda|$. First, $S_U \coloneqq \bigcup_{E \in U} E$, which need not be open in $|\Lambda|$; second, $W_U \coloneqq \bigcup_{E \in U}\left(\bigcup_{s \in E}|\Lambda|_{\leq s}\right)$ (see \ref{flabby-systems}), which is the smallest open subset of $|\Lambda|$ containing $S_U$. From the fact that $U$ is filtered it readily follows that so are $S_U$ and $W_U$. Now,
		\[
			\varprojlim u^*F\vert_U \cong \varprojlim F\vert_{S_U} \cong \varprojlim F\vert_{W_U}.
		\]
		The second isomorphism is obvious because $S_U$ is cofinal in $W_U$ by construction. The first follows in the same way as the first equality in (\ref{limit-finite-subsystems-maximum}).
		
		As a consequence, from the clearly commutative diagram
			\[
		\begin{tikzpicture}
			\matrix (m) [matrix of math nodes,
			row sep=1.6em, column sep=3em,
			text height=2ex,
			text depth=0.25ex]{
				\Gamma(|\pfinmax(\Lambda)|,\widetilde{u^*F}) = \varprojlim u^*F &  \Gamma(|\Lambda|,\tilde{F}) = \varprojlim F  \\
				\Gamma(U,\widetilde{u^*F}) = \varprojlim u^*F\vert_U & \Gamma(W_U,\tilde{F}) = \varprojlim F\vert_{W_U}  \\
			};
			\path[->, font=\scriptsize]
			(m-1-1) edge node[auto] {} (m-1-2)
			(m-1-1) edge node[auto] {} (m-2-1)
			(m-2-1) edge node[auto] {} (m-2-2)
			(m-1-2) edge node[auto] {} (m-2-2)
			;
		\end{tikzpicture}
		\]
		we see that because both horizontal maps are isomorphisms and the right vertical one is surjective (since $F$ is weakly flabby), the left one must also be surjective, which implies that $u^*F$ is weakly flabby.	
	\end{proof}
	
	\begin{cosa}
		\textbf{Recollection about $\R\varprojlim$.}
		Since it is going to be used repeatedly in the following results, we clarify that by $\R\varprojlim$ we mean the following: consider the category of left $R$-modules, $R\md$. For our purposes this is enough, but one could replace it by any other complete Grothendieck category $\SA$. The additive functor 
		\[
		\varprojlim \colon \cat(\pfinmax(\Lambda)^\op, R\md) \to R\md,
		\]
		can be extended (degree-wise) to the homotopy category. We keep denoting the composite
		\[
		\begin{tikzpicture}
			\matrix (m) [matrix of math nodes,
			row sep=1.6em, column sep=3em,
			text height=2ex,
			text depth=0.25ex]{
				\K(\cat(\pfinmax(\Lambda)^\op, R\md)) &  \K(R) \\
				&  \D(R) \\
			};
			\path[->, font=\scriptsize]
			(m-1-1) edge node[auto] {$\K(\varprojlim)$} (m-1-2)
			(m-1-2) edge node[auto] {$Q$} (m-2-2)
			(m-1-1) edge node[below] {$\varprojlim$} (m-2-2)
			;
		\end{tikzpicture}
		\]
		by $\varprojlim$, as shown, where $Q$ is the Verdier quotient. Then, $\R\varprojlim$ is just the right derived functor of $\varprojlim$
		\[
			\R\varprojlim \colon \D(\cat(\pfinmax(\Lambda)^\op, R\md))  \lto \D(R),
		\]
		which is triangulated and exists because $\cat(\pfinmax(\Lambda)^\op, R\md)$ is a Grothendieck category and so $K$-injective resolutions exist \cite[Theorem 5.4]{AJS}.
	\end{cosa}
	
	\begin{rem}
		The following results need some conditions on the cardinality of the inverse system in order to guarantee the possibility of using Goblot's theorem\footnote{Let $R$ be an associative unitary (possibly non-commutative) ring, $n \in \NN$ and $\Lambda$ a filtered poset of cardinality $\#\Lambda \leq \aleph_n$. If $F \colon \Lambda^\op \to R\md$ is any inverse system, then $\sideset{}{^{i}}\varprojlim F = 0$ for any $i \geq n+2$.} \cite{goblot}. We recall the notation $\aleph_\omega$ for the smallest cardinal bigger than all of $\aleph_n$ for $n \in \NN$. We will use it as a convenient notation, but no further cardinal theory will be used beyond what is already implicit in Goblot's statement.
	\end{rem}
	
	\begin{thm}\label{thm:quis-abelian-cat-ab4*-gen}
		Suppose $R\md$ is the category of left $R$-modules for a ring $R$, and let $\Lambda$ be a filtered poset of cardinality \,$\#\Lambda < \aleph_\omega$. If $F \colon \Lambda^\op \to \CCC(R)$ is an inverse system which is weakly flabby (degree-wise), there exists a quasi-isomorphism 
		\[ \varprojlim F \to \holim{} F.\]
	\end{thm}
	\begin{proof}
		First, from (\ref{limit-finite-subsystems-maximum}) we learn that $\varprojlim F \cong \varprojlim_{E \in \pfinmax(\Lambda)}F(\mu_E)$, where $\mu_E$ is the maximum of $E$ (equivalently, $\varprojlim = \varprojlim \circ u^*$ is an isomorphism of functors\footnote{Recall that this meant $\varprojlim_{\Lambda} = \varprojlim_{\pfinmax(\Lambda)} \circ u^*$ as functors $\Lambda^\op \to R\md$.}, where $u^*$ is defined as in (\ref{limit-finite-subsystems-u*})). On the other hand, Lemma \ref{lemma:holim-as-lim} yields
		\[
			\holim{} F \cong \varprojlim_{E \in \pfinmax(\Lambda)} \!\!\! \holim{}F\vert_E.
		\]
	
		For each $E \in \pfinmax(\Lambda)$, let us denote by 
		\[
			\varphi_E \colon \!\!\!\holim{} F\vert_E \to F(\mu_E)
		\]
		the morphism in $\CCC(R)$ constructed in the proof of Proposition \ref{prop:homotopy-equiv-finite-system}, which was shown to be a homotopy equivalence, and by $\cone(\varphi_E)$ the cone of $\varphi_E$. Thus, $\cone(\varphi_E)$ is exact for each $E \in \pfinmax(\Lambda)$ and we have semi-split exact sequences in $\CCC(R)$
		\begin{equation}\label{short-exact-sequences}
				0 \lto F(\mu_E) \lto \cone(\varphi_E) \lto \!\!\!\holim{} F\vert_E [1] \lto 0.
		\end{equation}
		Denote by $u^*F$, $C$ and $H[1]$ the respective inverse systems $\pfinmax(\Lambda)^\op \to \CCC(R)$, \textit{i.e.}, $u^*F$ is the one given by $E \mapsto F(\mu_E)$, $C$ the one defined by $E \mapsto \cone(\varphi_E)$, and $H[1]$ the one given by $E \mapsto \!\!\!\holim{} F\vert_E [1]$. From Lemma \ref{lem: flabby-over-subsets-with-maximum} we learn that $u^*F$ is degree-wise weakly flabby. Moreover, one can check that the system $H[1]$ is in fact flabby (degree-wise) with no hypotheses on $F$, thus weakly flabby. These two facts and the exact sequences (\ref{short-exact-sequences}) let us deduce that the system $C$ has the same property. Indeed, we know that weakly flabby inverse systems in $R\md$ have vanishing higher derived limits by Theorem \ref{thm: jensen-weakly-flabby}. Also, if we let $U$ be any filtered open susbet of $\pfinmax(\Lambda)$, note that since $u^*F$ and $H[1]$ are weakly flabby, so are their restrictions $u^*F\vert_U$ and $H[1]\vert_U$. Therefore, taking limits over the objects $E \in U$ in the exact sequences (\ref{short-exact-sequences}) we obtain the short exact sequence
		\[
			0 \lto \varprojlim u^*F\vert_U \lto \varprojlim C\vert_U \lto \varprojlim H[1]\vert_U \lto 0,
		\]
		because in each degree $n \in \ZZ$ the long exact sequence yields, by weak flabbinness of $u^*F\vert_U$,
		\[
			\sideset{}{^{1}}\varprojlim (u^*F\vert_U)^n = \sideset{}{^{1}}\varprojlim (u^*F)^n\vert_U = 0.
		\]
		As a consequence, we get the diagram, for each filtered open subset $U \subset \pfinmax(\Lambda)$
		\[
		\begin{tikzpicture}
			\matrix (m) [matrix of math nodes,
			row sep=1.6em, column sep=3em,
			text height=2ex,
			text depth=0.25ex]{
				0 & \varprojlim u^*F & \varprojlim C & \varprojlim H[1] & 0\\
				0 & \varprojlim u^*F\vert_U & \varprojlim C\vert_U & \varprojlim H[1]\vert_U & 0,\\
			};
			\path[->, font=\scriptsize]
			(m-1-1) edge node[auto] {} (m-1-2)
			(m-1-2) edge node[auto] {} (m-1-3)
			(m-1-3) edge node[auto] {} (m-1-4)
			(m-1-4) edge node[auto] {} (m-1-5)
			(m-2-1) edge node[auto] {} (m-2-2)
			(m-2-2) edge node[auto] {} (m-2-3)
			(m-2-3) edge node[auto] {} (m-2-4)
			(m-2-4) edge node[auto] {} (m-2-5)
			(m-1-2) edge node[auto] {} (m-2-2)
			(m-1-3) edge node[auto] {} (m-2-3)
			(m-1-4) edge node[auto] {} (m-2-4);
		\end{tikzpicture}
		\]
		and so the left and right vertical morphisms being surjections implies that so is the middle vertical one by the weak four lemma \cite[Chapter XII, Lemma 3.1.]{M}. That is, $C$ is weakly flabby.
		
		The short exact sequence for $U = \pfinmax(\Lambda)$ and (\ref{limit-finite-subsystems-u*}) give the existence of a distinguished triangle
		\begin{equation}\label{dist-triangle-holim-lim}
			\varprojlim_{E \in \pfinmax(\Lambda)}\!\!\!\holim{} F\vert_E \overset{\varphi}{\lto} \varprojlim F \lto \varprojlim_{E \in \pfinmax(\Lambda)}\cone(\varphi_E) \overset{+}{\lto}
		\end{equation}	 
		in $\D(R)$, where $\varphi = \varprojlim \varphi_E$, and it follows that $\cone(\varphi) \cong \varprojlim_{E \in \pfinmax(\Lambda)}\cone(\varphi_E)$ in $\D(R)$. Then, it is enough to show that $\varprojlim_{E \in \pfinmax(\Lambda)}\cone(\varphi_E)$ is exact. To do that, the first step is to note that, in the same way, the short exact sequences (\ref{short-exact-sequences}) give rise to distinguished triangles
		\begin{equation}\label{dist-triangles-each-finite-subset}
			\holim{} F\vert_E \overset{\varphi_E}{\lto} F(\mu_E) \lto \cone(\varphi_E) \overset{+}{\lto} 
		\end{equation}
		in $\D(R)$ in which the cone is isomorphic to zero, that is, the first arrow is an isomorphism. We claim that the triangle (\ref{dist-triangle-holim-lim}) is actually the image through $\R\varprojlim$ (a functor that preserves quasi-isomorphisms because it is triangulated) of the triangles (\ref{dist-triangles-each-finite-subset}), which automatically lets us conclude. 
		
		Indeed, the point is that since $\#\Lambda \leq \aleph_n$ for some $n \in \NN$, then $\# \pfinmax(\Lambda) \leq \aleph_n$ as well, and so by Goblot's result \cite{goblot} (one can also consult \cite[Théorème 3.1.]{jensen}) we have that for \emph{any} functor $G \colon \pfinmax(\Lambda)^\op \to R\md$, then $\sideset{}{^{i}}\varprojlim G = 0$ for all $i \geq n+2$. In other words, $\h^i(\R\varprojlim G) = 0$ for all $i \geq n+2$. Therefore, we have one of the equivalent conditions of (the dual of) \cite[Proposition 2.7.5.]{yellow}---which we can apply because the functor $\varprojlim \colon \cat(M(\Gamma)^\op, R\md) \to R\md$ is left-exact and every object in $\cat(\pfinmax(\Lambda)^\op, R\md)$ can be embedded into a $\varprojlim$-right-acyclic one (for instance, we can consider flabby inverse systems)---whence we get the conclusions of \emph{loc.cit}. In particular, a complex of $\varprojlim$-right-acyclic objects is a $\varprojlim$-right-acyclic complex. 
		
		Since weakly flabby inverse systems are right-acyclic for $\varprojlim$ (for example, they satisfy the conditions (\romannumeral 1)-(\romannumeral 3) on \cite[Proposition 2.7.2.]{yellow}, as shown in \cite[\S 1]{jensen}), the objects $u^*F$, $C$ and $H[1]$ are degree-wise right-acyclic, thus right-acyclic for $\varprojlim$. 
		
		Then, the canonical arrow $\varprojlim u^*F \to \R\varprojlim u^*F$ is an isomorphism in $\D(R)$, and the analogous conclusion holds for $C$ and $H[1]$, which amounts to saying that the distinguished triangle (\ref{dist-triangle-holim-lim}) is isomorphic to the distinguished triangle
		\begin{equation}\label{rlim-triangle}
			\R\varprojlim H \lto \R\varprojlim u^*F \lto \R\varprojlim C \overset{+}{\lto},
		\end{equation}
		so $\cone(\varprojlim \varphi_E)$ is exact (because $\R\varprojlim C$ is exact) and we are done.
	\end{proof}

	\begin{rem}
		Sequences that satisfy degree-wise the strong Mittag-Leffler condition introduced in \ref{flabby-systems} are examples of inverse systems for which Theorem \ref{thm:quis-abelian-cat-ab4*-gen} can be applied, as long as the sequence is of length $\gamma$, for $\gamma$ an ordinal whose cardinal is strictly smaller than $\aleph_\omega$. After all, those inverse systems in $R\md$ are flabby by Lemma \ref{lem: sequences-flabby-neeman}. 
		
		Also note that in the case of sequences weak flabbiness is equivalent to flabiness, since every open subset must be filtered.
	\end{rem}
	
	We can say even more about the homotopy limit in the category $R\md$, making the same assumptions regarding the cardinality of $\Lambda$:
	
	\begin{thm}\label{thm:in-rmod-holim-is-rlim}
		Suppose $R\md$ is the category of left $R$-modules for a ring $R$, and let $\Lambda$ be a filtered poset of cardinality \,$\#\Lambda < \aleph_\omega$. Let $F \colon \Lambda^\op \to \CCC(R)$ be \emph{any} inverse system. Then, there exists a natural isomorphism in $\D(R)$ 
		\[
			\holim{} F \cong \R\varprojlim F.
		\]
	\end{thm}
	
	\begin{rem}
		Observe that, in the conditions on which it is stated, Theorem \ref{thm:in-rmod-holim-is-rlim} tells us that
		\[
			\holim{} \!\!\!\cong \R\varprojlim \colon \D(\cat(\Gamma^\op, R\md)) \lto \D(R),
		\]
		that is, the homotopy limit is actually a triangulated functor between the appropriate derived categories.
	\end{rem}
	\begin{proof}
		This proof is based on that of Theorem \ref{thm:quis-abelian-cat-ab4*-gen}, and so the assertion will follow from the following additional two facts. The first one, which has already been noted in the proof of Theorem \ref{thm:quis-abelian-cat-ab4*-gen}, states that the inverse system 
		\[
			H \colon \pfinmax(\Lambda)^\op \lto \CCC(R\md), \quad E \mapsto \!\!\!\holim{} F\vert_E
		\]
		is always flabby (degree-wise) and so, reasoning in exactly the same way (noting that $\#\pfinmax(\Lambda) < \aleph_n$ for the same $n$ as $\#\Lambda$ and using the dual of \cite[Proposition 2.7.5.]{yellow} and Lemma \ref{lemma:holim-as-lim}), we find that the system $H$ is $\varprojlim$-right-acyclic, whence the canonical arrow
		\begin{equation}\label{holim-iso-rlimH}
			\holim{} F \cong \varprojlim H \lto \R\varprojlim H
		\end{equation}
		is an isomorphism in $\D(R)$.
		
		The second fact is that the derived version of (\ref{limit-finite-subsystems-u*}) holds, that is,
		\[
			\R\varprojlim_{\Lambda} \cong \R\varprojlim_{\pfinmax(\Lambda)} \circ u^*.
		\]
		Indeed, first note that since $\cat(\Lambda^\op, R\md)$ and $\cat(\pfinmax(\Lambda)^\op,R\md)$ are Grothendieck categories, the right derived functors of $\varprojlim_{\pfinmax(\Lambda)}$ and $u^*$ exist \cite[Theorem 5.4.]{AJS}. In fact, since $u^*$ is obviously exact, $\R u^*=u^*$. So, by \cite[Corollary 2.2.7.]{yellow}, the assertion comes down to showing that for any complex $M \in \K(\cat(\Lambda^\op, R\md))$ we can find a quasi-isomorphism $M \to I_M$ so that $I_M$ is $u^*$-right-acyclic and $u^*I_M$ is $\varprojlim$-right-acyclic. 
		
		In order to do that, observe that since $R\md$ satisfies $\mathsf{AB4^*}$ and has enough injectives, so does $\cat(\Lambda^\op, R\md)$ (for the claim regarding injectives, see for instance \cite[Lemma A.4.3.]{Ntc}), thus we can choose $I_M$ to be a $K$-injective resolution with the property that it is also injective in every degree (see \cite[Application 2.4.]{BN}). Now, the exact isomorphism of categories between $\cat(\Lambda^\op,R\md)$ and sheaves of $R$-modules on the topological space $|\Lambda|$ of \ref{flabby-systems} implies that, in each degree $n \in \ZZ$, the associated sheaf $\widetilde{I_M^n}$ is injective. Injective sheaves are flabby, so from Lemma \ref{lem: flabby-over-subsets-with-maximum} we learn that $u^*I_M$ is degree-wise weakly flabby. Again, owing to the fact that $\#\pfinmax(\Lambda) < \aleph_n$ and arguing as in the proof of Theorem \ref{thm:quis-abelian-cat-ab4*-gen}, we obtain from the dual of \cite[Proposition 2.7.5.]{yellow} that $u^*I_M$ is $\varprojlim$-right-acyclic, as desired.
		
		Finally, recall the distinguished triangle (\ref{rlim-triangle})
		\[
		\R\varprojlim H \lto \R\varprojlim u^*F \lto \R\varprojlim C \overset{+}{\lto}
		\]
		obtained in the proof of Theorem \ref{thm:quis-abelian-cat-ab4*-gen}. As remarked there, $\R\varprojlim C = 0$ in $\D(R)$, so the first map is an isomorphism. This combines with the consequences from the two facts from above to prove the claim. The first fact yielded the isomorphism (\ref{holim-iso-rlimH}) in $\D(R)$ between the homotopy limit and the first term of the triangle, and the second one tells us that the second term of the triangle is isomorphic to $\R\varprojlim F$, which lets us establish the claimed isomorphism.
	\end{proof}
	
	\section{Duality between homotopy limits and colimits}\label{Section:duality}
	The definition of homotopy limits was clearly dual from the definition of homotopy colimits \cite[\S 2]{AJS}. The goal of this short section is exploring the way this duality relates both constructions and the consequences we can obtain when combining said relationship with the main results of Section \ref{Section:homotopy-limit}.
	
	\begin{thm}\label{hom-and-hocolim-holim}
		Let $\SA$ be an abelian category satisfying $\mathsf{AB3}$ and $\Lambda$ a filtered poset. Then, for any direct system $G \colon \Lambda \to \CCC(\SA)$, it holds that for any $X \in \CCC(\SA)$
		\[
			\Hom^\bullet(\!\!\!\hocolim{}G, X) \cong \!\!\!\holim{} \Hom^\bullet(G(-), X).
		\]
	\end{thm}
	\begin{proof}
		
		At the level of objects, the isomorphism is clear. Simplifying notation and only writing $\textbf{s}$ where it should be $\textbf{s} \in \CN'_k(\Gamma)$, for any $n \in \ZZ$ we have
		\begin{equation*}
			\begin{aligned}
				&\Hom^n(\!\!\!\hocolim{}G, X) = \prod_{j \in \ZZ} \Hom((\!\!\!\hocolim{}G)^j, X^{j+n}) \\
				&= \prod_{j \in \ZZ} \Hom\left(\coprod_{k \geq 0}B^{\coprod}(G)^{-k \,\, j+k}, X^{j+n}\right) \cong \prod_{j \in \ZZ} \prod_{k \geq 0} \Hom\left(B^{\coprod}(G)^{-k \,\, j+k}, X^{j+n}\right) = \\
				&=\prod_{j \in \ZZ} \prod_{k \geq 0} \Hom\left(\coprod_{\textbf{s}}G(s_0)^{j+k}, X^{j+n}\right) \cong \prod_{j \in \ZZ} \prod_{k \geq 0} \prod_{\textbf{s}}	\Hom\left(G(s_0)^{j+k}, X^{j+n}\right)	\cong \\
				&\cong \prod_{k \geq 0} \prod_{\textbf{s}} \prod_{j \in \ZZ} \Hom\left(G(s_0)^{j+k}, X^{j+n}\right) = \prod_{k \geq 0} \prod_{\textbf{s}} \Hom^{n-k}(G(s_0), X) = \\
				&= \prod_{k \geq 0} B^{\prod}(\Hom^\bullet(G(-), X))^{k \, \, n-k} = \left(\!\!\!\holim{} \Hom^\bullet(G(-),X)\right)^n.
			\end{aligned}
		\end{equation*}
		Note that we have only made use of definitions, the universal property of the coproduct and the fact that products commute with each other, because all limits do.
		
		As for the differential, our goal is to check that
		\[
			d_{\Hom^\bullet}^n \colon \prod_{j \in \ZZ} \Hom\left(\coprod_{k \geq 0}B^{\coprod}(G)^{-k \,\, j+k}, X^{j+n}\right) \lto \prod_{j \in \ZZ} \Hom\left(\coprod_{k \geq 0}B^{\coprod}(G)^{-k \,\, j+k}, X^{j+n+1}\right)
		\]
		and
		\[
			d_{\text{holim}}^n \colon \prod_{k \geq 0} B^{\prod}\left(\Hom^\bullet(G(-), X)\right)^{k \,\, n-k} \lto \prod_{k \geq 0} B^{\prod}\left(\Hom^\bullet(G(-), X)\right)^{k \,\, n-k+1}
		\]
		are the same. In order to do it, we will check that, after applying the suitable isomorphisms, as above, so that the objects involved in both differentials are the same (namely, $\prod_{j \in \ZZ} \prod_{k \geq 0} \prod_{\textbf{s}}\Hom(G(s_0)^{j+k}, X^{j+n})$ for the degree $n$ object, and so on), composition with the projections yields the same morphism. The notation for the projections is the obvious one, \textit{i.e.}, we denote by $p_j$, $p_k$ and $p_{\textbf{s}}$ the obvious morphisms from
		\[
			\prod_{j \in \ZZ} \prod_{k \geq 0} \prod_{\textbf{s}}\Hom(G(s_0)^{j+k}, X^{j+n}) \cong \prod_{k \geq 0} \prod_{\textbf{s}} \prod_{j \in \ZZ} \Hom(G(s_0)^{j+k}, X^{j+n})
		\]
		in the corresponding factor. Explicitly, the codomain of the triple composite of the projections (in any order)  $p_j$, $p_k$ and $p_{\textbf{s}}$ will be
		\[
			\Hom(G(s_0)^{j+k}, X^{j+n}),
		\] 
		where $s_0$ is the first object in the chain of morphisms $\textbf{s}$.
		
		In the case of the first differential, if we denote by $d_{\text{hocolim}}^n$ the $n$-th differential of the homotopy colimit of $F$, we have 
		\begin{equation*}
			\begin{aligned}
				&p_j d_{\Hom^\bullet}^n = \Hom\left(d_{\text{hocolim}}^j, X^{j+n+1}\right) p_{j+1} + (-1)^{n+1} \Hom\left(\!\!\!\hocolim{}G^j, d_X^{j+n}\right) p_j. \\
			\end{aligned}	
		\end{equation*}
		
		By the isomorphism
		\[
			\Hom\left(\coprod_{k \geq 0}\coprod_{\textbf{s}}G(s_0)^{j+k}, X^{j+n}\right) \cong \prod_{k \geq 0}\prod_{\textbf{s}}\Hom\left(G(s_0)^{j+k}, X^{j+n}\right)
		\]
		 and the description of the differential of the homotopy colimit (see \cite[\S 2, p. 231]{AJS}) we can rewrite the equality above, after postcomposing with the two new projections we get as a consequence of the isomorphism, as 
		\begin{equation}\label{first-differential}
			\begin{aligned}
				p_{\textbf{s}} p_k p_j d_{\Hom^\bullet}^n =& \Hom\left(G(s_0 \to s_1)^{j+k}, X^{j+n+1}\right) p_{\widehat{\textbf{s}}_0}p_{k-1} p_{j+1} \\
				&+ \sum_{i=1}^{k} (-1)^i \Hom\left(\id_{G(s_0)^{j+k}}, X^{j+n+1}\right) p_{\widehat{\textbf{s}}_i} p_{k-1} p_{j+1} \\
				&+ (-1)^{-k} \Hom\left(d_{G(s_0)}^{j+k}, X^{j+n+1}\right) p_{\textbf{s}} p_k p_{j+1} \\			
				&+(-1)^{n+1} \Hom\left(G(s_0)^{j+k}, d_X^{j+n}\right) p_{\textbf{s}} p_k p_{j}. \\
			\end{aligned}	
		\end{equation}
		
		One must note the following notational fact about the contributions coming from the term $\Hom(d_{\text{hocolim}}^j, X^{j+n+1}) p_{j+1}$ to avoid confusion, specifically the first two summands of (\ref{first-differential}). Observe that after fixing $j \in \ZZ$, $k \geq 0$ and $\textbf{s} \in \CN'_k(\Lambda)$ and projecting, the codomain of the composite is 
		\[
			\Hom\left(G(s_0)^{j+k}, X^{j+n+1}\right).
		\]
		The degree in $G(s_1)^{j+k}$ and $G(s_0)^{j+k}$ in the first two summands of in (\ref{first-differential}) is, in fact, $(j+1)+(k-1)$, but for clarity we have chosen not to write it, since this is already reflected in the projections.
		
		As for the second differential, that is $d_{\text{holim}}^n$, note that it is the differential of a homotopy limit, as we have described at the beginning of the section. Thus, we can write it as
		\[
			\prod_{k \geq 0} \left(\prod_{\textbf{s}} \Hom^{n-k}\left(G(s_0), X\right)\right) \xrightarrow{\sum_{k} d_h^{k n-k} + (-1)^k d_v^{kn-k}} \prod_{k \geq 0} \left(\prod_{\textbf{s}} \Hom^{n-k+1}(G(s_0), X)\right).
		\]
		Recall that 
		\[
		p_k d_{\text{holim}}^n =  d_h^{k-1 n-k+1} p_{k-1} + (-1)^{k} d_v^{k n-k} p_k.
		\]
		Next, write, for any $n \in \ZZ$,
		\[
		\prod_{k \geq 0} B^{\prod}\left(\Hom^\bullet(G(-), X)\right)^{k n-k} \cong \prod_{k \geq 0} \prod_{\textbf{s}} \prod_{j \in \ZZ} \Hom\left(G(s_0)^{j+k}, X^{j+n}\right).
		\]
		This is the description we will use at the level of objects to describe the differential. It is, up to commuting products, the same description for objects we used for the first differential. Therefore, we will show that up to commutation of the compositions with the projections, this differential and the first one agree. 
		
		Again, by the knowledge of the description of both horizontal and vertical differentials, we can describe both and get the full expression from the equality above. So, after applying the isomorphism, we observe that the horizontal contribution gives
		\begin{equation}\label{horizontal-contribution}
			\begin{aligned}
				p_j  p_{\textbf{s}} d_h^{k-1 n-k+1} p_{k-1} =& \Hom\left(G(s_0 \to s_1)^{j+k}, X^{j+n+1}\right) p_{j+1} p_{\widehat{\textbf{s}}_0} p_{k-1} \\
				&+ \sum_{i=1}^{k} (-1)^i \id_{\Hom(G(s_0)^{j+k}, X^{j+n+1})} p_{j+1} p_{\widehat{\textbf{s}}_i} p_{k-1}, \\
			\end{aligned}
		\end{equation}
		whereas the vertical one yields 
		\begin{equation}\label{vertical-contribution}
			\begin{aligned}
				p_j  p_{\textbf{s}} (-1)^k d_v^{k n-k} p_{k} =& p_j (-1)^k d_{\Hom^\bullet(G(s_0),X)}^{n-k} p_{\textbf{s}} p_{k} \\
				=& (-1)^k (-1)^{n-k+1} \Hom\left(G(s_0)^{j+k},d_X^{j+n}\right) p_j p_{\textbf{s}} p_{k} \\
				&+ (-1)^k  \Hom\left(d_{G(s_0)}^{j+k}, X^{j+n+1}\right) p_{j+1} p_{\textbf{s}} p_{k}. \\
			\end{aligned}
		\end{equation}
		Of course, $\Hom\left(\id_{G(s_0)^{j+k}}, X^{j+n+1}\right) = \id_{\Hom(G(s_0)^{j+k}, X^{j+n+1})}$, so combining (\ref{horizontal-contribution}) and (\ref{vertical-contribution}) we obtain the same result as we did in the first differential in (\ref{first-differential}). Therefore both complexes in the statement are isomorphic in $\CCC(\SA)$ to the same complex, and so they are isomorphic themselves.
	\end{proof}
	
	\begin{cor}\label{cor:derived-limits-and-colimits}
		Suppose $R\md$ is the category of left $R$-modules for a ring $R$, and let $\Lambda$ be a filtered poset of cardinality \,$\#\Lambda < \aleph_\omega$. For any direct system $G \colon \Lambda \to \CCC(R)$ and $X \in \CCC(R)$, we have the following isomorphism in $\D(R)$
		\[
			\rhom^\bullet(\LL\varinjlim G, X) \cong \R\varprojlim \rhom^\bullet(G(-),X),
		\]
		or, equivalently,
		\[
		\rhom^\bullet(\!\!\!\hocolim{}\!G, X) \cong \holim{} \!\!\!\rhom^\bullet(G(-),X).
		\]
	\end{cor}
	\begin{proof}
		Because $R\md$ is a Grothendieck category, filtered colimits are exact, so the functor $\LL\varinjlim \colon \D(\cat(\Lambda, R\md)) \to \D(R)$ is actually just $\varinjlim$.
		
		On the right side term, we want to compute $\rhom^\bullet(G(-),X)$, but this can be done by taking a $K$-injective resolution on the second argument (see the Remark below). Therefore, we take a $K$-injective resolution $X \to I_X$ and obtain 
		\begin{equation*}
			\begin{aligned}
				\rhom^\bullet(\LL\varinjlim G, X) &= \rhom^\bullet(\varinjlim G, X) = \Hom^\bullet(\varinjlim G, I_X) \\
				&\cong \Hom^\bullet(\!\!\!\hocolim{}\!G, I_X) 
				\cong \!\holim{} \!\!\!\Hom^\bullet(G(-), I_X) \\
				&\cong \R\varprojlim \Hom^\bullet(G(-), I_X) \cong \R\varprojlim \rhom^\bullet(G(-),X).
			\end{aligned}
		\end{equation*}
		The isomorphism from the first line to the second is just \cite[Theorem 2.2.]{AJS}, while the one in the second line is Theorem \ref{hom-and-hocolim-holim}. Finally, the one from the second to the third line is Theorem \ref{thm:in-rmod-holim-is-rlim}, and the final one is the discussion above.
	\end{proof}

	\begin{rem}
		For the more careful reader, a brief comment can be made about the computation of $\rhom^\bullet(G(-),X)$ above. On the one hand, we can see this as right-deriving the functor
		\[
		\Hom^\bullet(-,X) \colon \cat(\Lambda,\CCC(R)) \to \cat(\Lambda^\op, \CCC(R)),
		\] 
		for whose definition one only must note that if $s \leq t$ in $\Lambda$, which implies we have a map $G(s) \to G(t)$ in $\CCC(R)$, we obviously obtain a map
		\[
		\Hom^\bullet(G(-),X)(t) = \Hom^\bullet(G(t),X) \to \Hom^\bullet(G(-),X)(s) = \Hom^\bullet(G(s),X).
		\]
		Alternatively, we can see it as deriving the functor
		\[
		\Hom^\bullet(G,-) \colon \CCC(R) \to \cat(\Lambda^\op, \CCC(R)),
		\]
		which is defined analogously. 
		
		So, in order to compute $\rhom^\bullet(G(-),X)$, we could take a $K$-projective resolution $P_G \to G$ in $\CCC(\cat(\Lambda, R\md)) \cong \cat(\Lambda, \CCC(R))$ (which exists because the category $\cat(\Lambda, R\md)$ has enough projectives and satisfies $\mathsf{AB4}$) and compute $\Hom^\bullet(P_G,X)$, or take a $K$-injective resolution $X \to I_X$ and compute $\Hom^\bullet(G,I_X)$. The fact that both approaches yield quasi-isomorphic complexes (of diagrams) follows from the fact that quasi-isomorphisms are checked pointwise.
	\end{rem}

\section{Homotopy limits and colocalizing subcategories}\label{Section:homotopy-limits-and-colocalizing-subcategories}
	In this section we present the analogous to the main result in \cite[\S 3]{AJS}, with similar techniques, albeit taking into account the particularities of working with limits. We begin recalling the basic definition needed to state the main result:
	
	Let $\ST$ be a triangulated category with products. A triangulated subcategory $\SC$ of $\ST$ is said to be \textit{colocalizing} if it is closed under products. The main result, Theorem \ref{thm:holim-colocalizing}, says that colocalizing subcategories are stable under homotopy limits. The idea is to see the homotopy limit as a countable limit with pleasant properties. 

	We say that a complex $M \in \CCC(A)$ possesses a \textit{exhaustive cofiltration} if there exist a family of complexes $\{P_n\}_{n \in \NN}$ such that for each $n \in \NN$ there is an epimorphism $P_{n+1} \epi P_n$ and so that $M$ is the limit in $\CCC(\SA)$ of the diagram determined by the family, that is,
	\[
		\varprojlim_{n \in \NN} P_n = M.
	\]

	\begin{lem}\label{lemma:totalization-products-filtration}
		Let $i_0 \in \ZZ$ be a fixed integer, $\SA$ be an abelian category satisfying $\mathsf{AB3^*}$ and $\{X,d_h,d_v\}$ a double complex such that $X^{i\bullet} = 0$ for all $i < i_0$. Then, $\Tot^{\prod}(X)$ possesses a semi-split, exhaustive cofiltration of complexes  $\{P_n\}_{n \in \NN}$ such that $P_n/X^{n\bullet} [-n] \cong P_{n-1}$ for each $n \in \NN$. 
	\end{lem}
	\begin{proof}
		By taking desuspensions we may assume that $i_0=0$. Let $P_0 \coloneqq X^{0\bullet}$; the rest of the complexes in the family $\{P_n\}_{n \in \NN}$ will be constructed inductively. In fact, they will be constructed so that $P_n$ is the fiber in a distinguished triangle
		\begin{equation}\label{dist-triangle-filtration}
			P_n \xrightarrow{g_{n-1}} P_{n-1} \xrightarrow{h_{n-1}} 
		X^{n \bullet}[-n+1] \xrightarrow{\nu_n} P_n[1].
		\end{equation}
		Therefore, the complex $P_n$ will be, at the graded level, 
		\[
			P_n \cong P_{n-1} \oplus X^{n \bullet}[-n] \cong  P_{n-1} \times X^{n \bullet}[-n],
		\] 
		while $g_{n-1}$ is then just the projection onto $P_{n-1}$ and $\nu_n$ is just the inclusion.
		
		In order to construct $F_1$ we define $h_0 \colon F_0 = X^{0 \bullet} \to X^{1\bullet}$ to be $d_h^{0 \bullet}$, which is a morphism of complexes because of our convention regarding double complexes (see Section \ref{Section:Conventions}). For an arbitrary $n \geq 1$, we set
		$h_n \colon P_n \to X^{n+1 \bullet}[-n]$ to be the composite
		\[
			\begin{tikzpicture}
				\matrix (m) [matrix of math nodes,
				row sep=1.6em, column sep=5em,
				text height=2ex,
				text depth=0.25ex]{
					P_n = P_{n-1} \times X^{n\bullet}[-n] &  \\
				X^{n \bullet} [-n] &  	X^{n+1 \bullet} [-n],\\
				};
				\path[->, font=\scriptsize]
				(m-1-1) edge node[auto] {$h_n$} (m-2-2)
				(m-1-1) edge node[left] {$p_n$} (m-2-1)
				(m-2-1) edge node[below] {$d_h^{n \bullet}[-n]$} (m-2-2);
			\end{tikzpicture}
		\] with $p_n$ being the canonical projection. The composed morphism $h_n$ is a moprhism of complexes because both $p_n$ and $d_h^{n \bullet}[-n]$ are. 
		
		It is then clear that the morphisms $P_n \to P_{n-1}$ are semi-split morphisms of complexes, for in each degree $j \in \ZZ$
		\[
			P_n^j = \prod_{k=0}^{n} X^{k \,\, j-k} \lto P_{n-1}^j = \prod_{k=0}^{n-1} X^{k \,\, j-k}
		\]is the projection into the smaller product, so it is in fact a retraction. In particular, the arrows $P_n \to P_{n-1}$ are epimorphisms in $\CCC(\SA)$. The fact that $P_n/X^{n\bullet} [-n] \cong P_{n-1}$ follows from the semi-split short exact sequence 
		\[
		 0 \lto X^{n \bullet}[-n] \lto P_n \lto P_{n-1} \lto 0
		\]
		associated to the distinguished triangle (\ref{dist-triangle-filtration})
		and the description of the $P_n$ given above.
		
		Finally, we show that it is exhaustive, meaning that
		\[
			\varprojlim_{n \geq 0} P_n \cong \Tot^{\prod}(X)
		\]
		in $\CCC(\SA)$. Firstly, $\Tot^{\prod}(X)$ is a \textit{cone over the diagram} formed by the $P_n$, since we can define projections $\pi_n \colon \Tot^{\prod}(X) \to P_n$ for every $n \geq 0$ naturally as
		\[
			\pi_n^j \colon \Tot^{\prod}(X)^{j} = \prod_{k \geq 0} X^{k \,\,j-k} \lto \prod_{k=0}^{n} X^{k \,\,j-k} = P_n^j
		\]
		in each degree $j \in \ZZ$. Commutativity with the morphisms $g_{n-1} \colon P_n \to P_{n-1}$ is obvious. Inductively, we see that the morphisms $\pi_n$ are indeed morphisms of complexes for every $n \geq 0$. For $n=0$ it is a quite straightforward check. Suppose then $n > 0$ and $\Tot^{\prod}(X) \to P_{n-1}$ is a morphism of complexes. Denote by $d_{n}^j \colon P_{n}^j \to P_{n}^{j+1}$ the differential of the complex $P_n$ for any $n$. By writing, for any $j \in \ZZ$, the differential $d_n^j$ in terms of $d_{n-1}^j$ (because recall $P_n$ is the fiber of a morphism $P_{n-1} \to X^{n \bullet}[-n+1]$) it is again a straightforward computation to check that the equality 
		\[
			\pi_{n}^{j+1} d_{\Tot^{\prod}(X)}^j = d_{n}^j \pi_n^j
		\]
		holds. Therefore, $\Tot^{\prod}(X)$ is a cone over the diagram, and the construction shows it is quite clearly universal. After all, each object $P_n$, as we have defined them, is (as a complex, not just as a graded object) the totalization by products (or coproducts, since only finitely many objects are involved) of $X$ between horizontal degrees $0$ (because of the strict bounds of $X$) and $n$.
	\end{proof}
	
	\begin{rem}
		
		From the previous result we learn that the inverse system $\NN^\op \to \CCC(\SA)$ given by $n \mapsto P_n$ obtained in the proof satisfies the epimorphic Mittag-Leffler condition degree-wise. Suppose $H \colon \NN^\op \to \CCC(\SA)$ is any inverse system satisfying the epimorphic Mittag-Leffler condition degree-wise. If $\SA$ further satisfies $\mathsf{AB4^*}$, consider the two-term complex (of complexes) analogous to the one presented in \ref{countable-higher-derived-limits-abelian-cat}, that is,
		\[
			\prod_{i=0}^{\infty} H(i) \xrightarrow{1 - \text{shift}} \prod_{i=0}^{\infty} H(i),
		\]
		We remind the reader that the $\text{shift}$ map is the composite 
		\[
			\prod_{i=0}^{\infty} H(i) \xrightarrow{\pi} \prod_{i=1}^{\infty} H(i) \xrightarrow{\prod_{i=1}^{\infty} [H(i) \to H(i-1)]} \prod_{i=0}^{\infty} H(i),
		\]
		with $\pi$ being the canonical projection.
		
		Since every limit and colimit in the category $\CCC(\SA)$ is computed degree-wise,  denote, for any $m \in \ZZ$, the inverse system $H^m \colon \NN^\op \to \SA$ given by $n \mapsto H^m(n) = H(n)^m$. As mentioned, kernels and cokernels are computed degree-wise, and because in each degree $m \in \ZZ$ the map $1- \text{shift}$ is again $1 - \text{shift}$, we obtain that $\ker(1 - \text{shift})$ and $\cok(1 -  \text{shift})$ are the complexes in $\CCC(\SA)$ whose objects in degree $m \in \ZZ$ are given by $\varprojlim H^m$ and $\sideset{}{^{1}}\varprojlim H^m$, respectively. In each degree $m \in \ZZ$, we are just working with (higher) limits of inverse systems in $\SA$; namely, the diagrams $H^m \colon \NN^\op \to \SA$. Finally, assume that $\SA$ has a generator. In these conditions, we learn from Theorem \ref{thm:vanishing-countable-mittag-leffler-higher-limits} that for each $m \in \ZZ$ we have $\sideset{}{^{1}}\varprojlim H^m = 0$, so the complex $\cok(1 -  \text{shift})$ is zero. Of course, $\ker(1 - \text{shift}) = \varprojlim H$. This yields the short exact sequence	of complexes	
		\begin{equation}\label{short-exact-seq-milnor}
			0 \lto \varprojlim H \lto \prod_{i=0}^{\infty} H(i) \xrightarrow{1 - \text{shift}} \prod_{i=0}^{\infty} H(i) \lto 0,
		\end{equation}
		due to Milnor.
	\end{rem}
	
	\begin{lem}\label{lem:colocalzing-mittag-leffler-countable-system}
		Let $\SA$ be an abelian category with a generator and satisfying $\mathsf{AB3}$ and $\mathsf{AB4^*}$, and $\SC$ a colocalizing subcategory of $\D(\SA)$. Let $H \colon \NN^\op \to \CCC(\SA)$ be an inverse system satisfying the epimorphic Mittag-Leffler condition degree-wise. If $H(n) \in \SC$ for every $n \in \NN$, then $\varprojlim H \in \SC$.
	\end{lem}
	\begin{proof}
		First, note that $\SA$ satisfying $\mathsf{AB4^*}$ implies that $\D(\SA)$ has products and they coincide with those of $\K(\SA)$ and $\CCC(\SA)$. By the previous remark, we obtain the short exact sequence
		\[
			0 \lto \varprojlim H \lto \prod_{i=0}^{\infty} H(i) \xrightarrow{1 - \text{shift}} \prod_{i=0}^{\infty} H(i) \lto 0,
		\]
		whence a distinguished triangle
		\[
			\varprojlim H \lto \prod_{i=0}^{\infty} H(i) \xrightarrow{1 - \text{shift}} \prod_{i=0}^{\infty} H(i) \overset{+}{\lto}.
		\]
		Since both the middle and right terms of the triangle belong to $\SC$ because $\SC$ is colocalizing, so does the left term.
		
	\end{proof}
	
	\begin{thm}\label{thm:holim-colocalizing}
		Let $\SA$ be an abelian category with a generator and satisfying $\mathsf{AB3}$ and $\mathsf{AB4^*}$, $\SC$ a colocalizing subcategory of \,$\D(\SA)$ and $\Lambda$ a filtered poset. For any inverse system $F \colon \Lambda^\op \to \CCC(\SA)$ such that $F(s) \in \SC$ for all $s \in \Lambda$, the homotopy limit
		\[
		\holim{} F
		\]
		also belongs to $\SC$.
	\end{thm}
	\begin{proof}
		By definition, we have 
		\[
			\holim{} F = \Tot^{\prod}(B^{\prod}(F)).
		\]
		Again, by the definition of $B^{\prod}(F)$, we obtain that each column $B^{\prod}(F)^{i \bullet}$ belongs to $\SC$ because they are defined as products of the complexes $F(s)$, with $s \in \Lambda$, and $\SC$ is colocalizing. Further, Lemma \ref{lemma:totalization-products-filtration} tells us that 
		\[
			\Tot^{\prod}(B^{\prod}(F)) = \varprojlim_{n \in \NN} P_n.
		\]
		Inductively, we learn that each $P_n$ belongs to $\SC$. This is true for $P_0 = B^{\prod}(F)^{0 \bullet}$, and the triangles (\ref{dist-triangle-filtration}) yield $P_n \in \SC$ after assuming $P_{n-1} \in \SC$. The Remark after Lemma \ref{lemma:totalization-products-filtration} also informs us that the inverse system defined by the $P_n$ satisfies the epimorphic Mittag-Leffler condition degree-wise; therefore, the application of Lemma \ref{lem:colocalzing-mittag-leffler-countable-system} to $P \colon \NN^\op \to \CCC(\SA)$ defined by $n \mapsto P_n$ lets us conclude. 
	\end{proof}
	

\end{document}